\documentclass[a4paper,12pt]{amsart}
\usepackage{amscd, amssymb, pgf, tikz, hyperref}

\newcounter{Aidanscounter}
\newcommand{\saveenumi}{\setcounter{Aidanscounter}{\value{enumi}}}
\newcommand{\restoreenumi}{\setcounter{enumi}{\value{Aidanscounter}}}

\newcommand{\id}{\operatorname{id}}
\newcommand{\dom}{\operatorname{dom}}
\newcommand{\cod}{\operatorname{cod}}
\newcommand{\Obj}{\operatorname{Obj}}
\newcommand{\Hom}{\operatorname{Hom}}

\newcommand{\lsp}{\operatorname{span}}
\newcommand{\End}{\operatorname{End}}
\newcommand{\Aut}{\operatorname{Aut}}

\newcommand{\field}[1]{\mathbb{#1}}
\newcommand{\CC}{\field{C}}
\newcommand{\NN}{\field{N}}
\newcommand{\TT}{\field{T}}
\newcommand{\ZZ}{\field{Z}}

\newcommand{\Cc}{\mathcal{C}}
\newcommand{\Ee}{\mathcal{E}}
\newcommand{\Gg}{\mathcal{G}}
\newcommand{\Hh}{\mathcal{H}}
\newcommand{\Kk}{\mathcal{K}}
\newcommand{\Ll}{\mathcal{L}}
\newcommand{\Mm}{\mathcal{M}}
\newcommand{\Oo}{\mathcal{O}}
\newcommand{\Uu}{\mathcal{U}}
\newcommand{\Tt}{\mathcal{T}}

\newcommand{\corresp}[3]{%
    \ensuremath{\sideset{_{#1}}{_{#3}}{\operatorname{\mbox{$#2$}}}}
}%
\newcommand{\LCHS}{\mathcal{T}}
\newcommand{\Lmin}{\Lambda^{\operatorname{min}}}

\newcounter{mycounter}

\def\clsp{\overline{\operatorname{span}}}

\newtheorem{thm}{Theorem}[subsection]
\newtheorem{cor}[thm]{Corollary}
\newtheorem{lem}[thm]{Lemma}
\newtheorem{prop}[thm]{Proposition}

\theoremstyle{definition}
\newtheorem{dfn}[thm]{Definition}

\theoremstyle{remark}
\newtheorem{rmk}[thm]{Remark}

\newtheorem{example}[thm]{Example}
\newtheorem{examples}[thm]{Examples}
\newtheorem*{examples*}{Examples}

\newtheorem{ntn}[thm]{Notation}

\numberwithin{equation}{subsection}

\title[Graphs of $C^*$-correspondences]{Graphs of $C^*$-correspondences and Fell bundles}

\author{Valentin Deaconu}
\address{Valentin Deaconu, Alex Kumjian\\ Department of Mathematics (084)\\ University
of Nevada\\ Reno NV 89557-0084\\ USA} \email{vdeaconu, alex@unr.edu}
\author{Alex Kumjian}

\author{David Pask}
\address{David Pask, Aidan Sims\\ School of Mathematics and
Applied Statistics  \\
University of Wollongong\\
NSW  2522\\
AUSTRALIA} \email{dpask, asims@uow.edu.au}
\author{Aidan Sims}



\thanks{This research was supported by the Australian Research Council, the Fields Institute, and  BIRS}

\date{30 December 2008; revised 17 February 2009.}

\begin{document}
\begin{abstract}
We define the notion of a $\Lambda$-system of
$C^*$-correspondences associated to a higher-rank graph
$\Lambda$. Roughly speaking, such a system assigns to each
vertex of $\Lambda$ a $C^*$-algebra, and to each path in
$\Lambda$ a $C^*$-correspondence 
in  a way which carries compositions of paths to balanced tensor
products of $C^*$-correspondences. Under some simplifying
assumptions, we use Fowler's technology of Cuntz-Pimsner
algebras for product systems of $C^*$-correspondences to
associate a $C^*$-algebra to each $\Lambda$-system. We then
construct a Fell bundle over the 
path groupoid $\Gg_\Lambda$ and show that the $C^*$-algebra of the
$\Lambda$-system coincides with the reduced cross-sectional
algebra of the Fell bundle. We conclude by discussing several
examples of our construction arising in the literature.
%
\end{abstract}

\maketitle

\section{Introduction}

The Cuntz-Krieger algebras introduced in  \cite{CK1980} in 1980
were considered the $C^*$-analogues of type III factors, and
formed a bridge between symbolic dynamics and operator
algebras. These algebras have since been generalised in various
ways including discrete graph algebras \cite{KPRR},
Cuntz-Pimsner algebras \cite{Pim}, topological graph algebras
\cite{Katsura2004}, and higher rank graph algebras \cite{KP}.
Between them, these generalisations include large classes of
classifiable $C^*$-algebras, like
AF-algebras~\cite{Drinen2000},
A${\TT}$-algebras~\cite{PRRS2006}, crossed products by
${\ZZ}^k$~\cite{FPS, KP2003}, and Kirchberg algebras~\cite{RS,
Spielberg2007}.

In the current paper, we will be interested in a further
generalisation of this theory based on structures which we call
$\Lambda$-systems of $C^*$-correspondences, and on the
relationship between this construction and Fell bundles over
groupoids. Our construction contains elements of both
higher-rank graph $C^*$-algebras and of Cuntz-Pimsner algebras,
so we must digress a little to discuss both before describing
our results.

Building on the theory of graph $C^*$-algebras introduced in
\cite{EW, KPRR}, higher-rank graphs, or $k$-graphs, and their
$C^*$-algebras were developed in \cite{KP} to provide a
graph-based model for the Cuntz-Krieger algebras of Robertson
and Steger \cite{RS}. They have subsequently attracted
widespread research interest (see, for example,
\cite{pp_DY2008, Evans, FMY, Hopenwasser2005, KP2006,
pp_PP2008, PZ2005, SZ2008}).

A $k$-graph is a kind of $k$-dimensional graph, which one
visualises as a collection $\Lambda^0$ of vertices together
with $k$ collections of edges $\Lambda^{e_1}, . . . ,
\Lambda^{e_k}$ which we think of as lying in $k$ different
dimensions. As an aid to visualisation, we distinguish the
different types of edges using $k$ different colours. The
higher-dimensional nature of a $k$-graph is encoded by the
factorisation property which implies  that each path
consisting of two edges of different colours can be
re-factorised in a unique way with the order of the colours
reversed. When $k=1$, we obtain an ordinary directed graph, and
the definition of the higher-rank graph $C^*$-algebra
given by Kumjian and Pask then reduces to that of the
 graph $C^*$-algebra \cite{KPRR}. However, for $k
\ge 2$, there are many $k$-graph algebras which do not arise as
graph algebras. For example, the original work of Robertson and
Steger on higher-rank Cuntz-Krieger algebras describes numerous
$2$-graphs $\Lambda$ for which $C^*(\Lambda)$ is a Kirchberg
algebra and $K_1(C^*(\Lambda))$ contains torsion.

Cuntz-Pimsner algebras, originally introduced by Pimsner in
\cite{Pim} and since studied by many authors (see, for example,
\cite{pp_AA2005, BR2006, Carlsen2008, De2, FMR2003, FR1999, Katsura2004a, pp_KL2007, MS1998b, Sch}) are a common
generalisation of graph $C^*$-algebras and of crossed products
by $\ZZ$. Let $B$ be a $C^*$-algebra and let  $X$ be a right Hilbert $B$-module.
If a second $C^*$-algebra $A$ acts by adjointable operators on $X$, we refer to $X$ as a
$C^*$-correspondence from $A$ to $B$. The prototype arises from
a $C^*$-homomorphism $\phi : A \to B$; there is then a
$C^*$-correspondence ${_{\phi}B}$ from $A$ to $B$ which is
equal to $B$ as a Banach space, has inner product
$\langle x,y \rangle_B := x^*y$ for $x,y \in
{_\phi B}$, and has operations given by $a
\cdot x \cdot b = \phi(a)xb$ for $a \in A$, $x \in {_\phi B}$,
and $b \in B$. A $C^*$-correspondence from $A$ to itself is called
a $C^*$-correspondence over $A$. Pimsner's
construction associates to each $C^*$-correspondence over $A$ a
$C^*$-algebra $\Oo_X$ which (under mild hypotheses) contains an
isomorphic copy $i_A(A)$ of $A$ and an isometric copy $i_X(X)$
of $X$, and in which the operations in $X$ are implemented
$C^*$-algebraically.

In this paper we seek to combine elements of these two
constructions. Our  model is the situation of $\Gamma$-systems
of $k$-morphs introduced in \cite{KPS2} to unify various
constructions of higher rank graphs. A $k$-morph is the
analogue at the level of $k$-graphs of a $C^*$-correspondence
between $C^*$-algebras. For a precise definition, see Example \ref{ex:systems}(ii).

One of the main results of \cite{KPS2} is that there is a
category $\Mm$ whose objects are $k$-graphs and whose morphisms
are isomorphism classes of $k$-morphs, and there is a functor
$\big(\Lambda \mapsto C^*(\Lambda),  [W] \mapsto [\Hh(W)]\big)$
from $\Mm$ to the category $\Cc$ whose objects are
$C^*$-algebras and whose morphisms are isomorphism classes of
$C^*$-correspondences. Given an $\ell$-graph $\Gamma$, a
$\Gamma$-system of $k$-morphs is a collection $\{\Lambda_v :
v\in\Gamma^0\}$ of $k$-graphs connected by $k$-morphs
$\{W_\gamma:\gamma\in\Gamma\}$ satisfying appropriate
compatibility conditions. This gives rise to a family indexed
by $\gamma \in \Gamma$ of
$C^*(\Lambda_{r(\gamma)})$--$C^*(\Lambda_{s(\gamma)})$
$C^*$-correspondences $\Hh(W_\gamma)$. The $\Gamma$-system also
gives rise to a $(k+l)$-graph $\Sigma$ called the
$\Gamma$-bundle for the system, and this $\Sigma$ encodes all
the information in the $\Gamma$-system. The $C^*$-algebra
$C^*(\Sigma)$ contains isomorphic and mutually orthogonal
copies of the $C^*(\Lambda_v)$ and isometric copies of the
$\Hh(W_\gamma)$, so it has the flavour of a Cuntz-Pimsner
algebra for the system of correspondences arising from the
$\Gamma$-system. Indeed, when $\Gamma$ is the graph consisting
of just one vertex $v$ and one loop $e$, Theorem~6.8
of~\cite{KPS2} shows that $C^*(\Sigma) \cong \Oo_{\Hh(W_e)}$.

In this work, we generalize this idea, defining a system  $(A,X,\chi)$  of
$C^*$-correspondences over a $k$-graph $\Lambda$,
by associating a $C^*$-algebra $A_v$ to each vertex
$v\in\Lambda^0$, a $A_{r(\lambda)}$--$A_{s(\lambda)}$
$C^*$-correspondence $X_\lambda$ to each  $\lambda\in\Lambda$,
and a compatibility isomorphism $\chi_{\lambda,\mu} : X_\lambda
\otimes_{A_{s(\lambda)}} X_\mu \to X_{\lambda\mu}$ to each
composable pair $\lambda,\mu$ of paths. The case $k=1$ is the easiest,
since we may always take each $X_\lambda$ to be
$X_{\lambda_1}\otimes \cdots \otimes X_{\lambda_n}$, where
$\lambda=\lambda_1\cdots\lambda_n$ with $d(\lambda_i)=1$,
and  the compatibility
isomorphisms $\chi_{\lambda,\mu}$ to be the canonical ones. Under a number of
simplifying assumptions (see Definition~\ref{dfn:regular}), we
define a $C^*$-algebra $C^*(A,X,\chi)$ by first constructing
from $(A,X,\chi)$ a product system $Y$ of
$C^*$-correspondences over $\NN^k$, and then tapping into Fowler's theory of
Cuntz-Pimsner algebras for such product systems \cite{F99}.
Under the same simplifying hypotheses, we then construct a Fell
bundle $E_X$ over the graph groupoid $\Gg_\Lambda$ developed in
\cite{KP} such that $C^*_r(\Gg_\Lambda, E_X)$ and
$C^*(A,X,\chi)$ are isomorphic. Our main results are: a version
of the gauge-invariant uniqueness theorem for $C^*(A,X,\chi)$
(Theorem~\ref{thm:giut}); the construction of the Fell bundle
$E_X$ itself (Theorem~\ref{thm:fellout}); and the isomorphism
between $C^*(A,X,\chi)$ and the reduced cross-sectional algebra
$C^*_r(\Gg_\Lambda, E_X)$ (Theorem~\ref{thm:isomorphism}).

The relationship between between product systems of
$C^*$-correspondences and $k$-graphs was previously
investigated in \cite{RS2005}. There the authors show that
given a $k$-graph $\Lambda$ there is a product system
$Y_\Lambda$ over $\NN^k$ whose coefficient algebra is
$C_0(\Lambda^0)$ and whose fibre over $n$ is the completion of
$C_c(\Lambda^n)$ under an appropriate $C_0(\Lambda^0)$-valued
inner product (the $\Lambda^n$ are given the discrete
topology). This relates to $\Lambda$-systems as follows. Let
$X$ be the $\Lambda$-system with $X_\lambda = \CC$ for all
$\lambda$ and with the $\chi_{\lambda,\mu}$ determined by
multiplication. Then the product system $Y$ described in the
preceding paragraph is precisely the $Y_\Lambda$ arising in
\cite{RS2005}; in particular, it follows from Theorem~4.2 of
\cite{RS2005} that our $\Oo_Y$ coincides with $C^*(\Lambda)$ under
our regularity hypotheses, which in this situation just boil
down to the requirement that $\Lambda$ is row-finite and has no
sources.

Our construction is also consistent with the example of
$\Gamma$-systems of $k$-morphs. Given an $\ell$-graph $\Gamma$
and a $\Gamma$-system $W$ of $k$-morphs in the sense of
\cite{KPS2}, there is a $\Gamma$-system $X = (A_v , X_\gamma,
\chi )$ of $C^*$-correspondences, where $A_v = C^*(\Lambda_v)$
and $X_\gamma$ is equal to the $C^*$-correspondence
$\Hh(W_\gamma)$ constructed in \cite[Proposition~6.4]{KPS2}.
Moreover, $C^*(A,X,\chi) \cong C^* (\Sigma )$ where $\Sigma$ is
the $\Gamma$-bundle described above associated to the
$\Gamma$-system of $k$-morphs $W_\gamma$.

Our construction is quite general, and includes a number of
diverse situations studied by a variety of authors in recent
papers. We will briefly discuss here how our work relates to
four such situations; we present the details of these examples
as well as a number of others in Section~\ref{sec:examples}.

The first situation covered by our construction which we
mention here is Cuntz's study of twisted tensor products
\cite{Cu}. Cuntz considers a $C^*$-algebra $A\times_{\mathcal
U}{\mathcal O}_n$, where $A$ is a $C^*$-algebra, $\Oo_n$ is the
Cuntz algebra, and  $\Uu=(U_1,...,U_n)$ is a family of unitaries
implementing automorphisms $\alpha_i$ of $A$. This defines a
system of $C^*$-correspondences over the $1$-graph $B_n$ with
one vertex $v$ and $n$ loop-edges $e_1, \dots, e_n$ based at
$v$, whose $C^*$-algebra $C^*(B_n)$ is canonically isomorphic
to $\Oo_n$. The $C^*$-algebra which we associate to this $B_n$-system
coincides with Cuntz's twisted tensor product.

Our construction also generalises the situation considered in
Section~5.3 of \cite{PWY}. There Pinzari, Watatani and Yonetani
study KMS states on a $C^*$-algebra constructed from a family
of compatible $C^*$-correspondences $X_{i,j}$, one for each non-zero
entry in  a finite
$\{0,1\}$-matrix
$\Sigma=(\sigma_{i,j})\in M_n(\{0,1\})$. The $C^*$-algebra
they associate to these data is the Cuntz-Pimsner algebra of
$\bigoplus_{\sigma_{i,j} = 1} X_{i,j}$. Let $\Lambda$ be the $1$-graph
with a vertex $v_i$ for
each $1 \le i \le n$ and an edge $e_{i,j}$ from $v_j$ to $v_i$
if and only if $\sigma_{i,j} = 1$. Then
$X_{e_{i,j}} := X_{i,j}$ determines a $\Lambda$-system
 of $C^*$-correspondences $(A,X,\chi)$. By construction, our $C^*(A,X,\chi)$
 coincides with the
Cuntz-Pimsner algebra studied by Pinzari-Watatani-Yonetani.

A third situation related to our work is that of Ionescu,
Ionescu-Watatani, and Quigg \cite{Ionescu2006, Ion, pp_IW2004,
Q} on $C^*$-algebras associated to Mauldin-Williams graphs. In
the setting studied by Ionescu \cite{Ion}, a Mauldin-Williams graph
consists of a directed graph $\Lambda$,
compact metric spaces $T_v$ associated to the vertices of $\Lambda$, and strict
contractions $\varphi_e : T_{s(e)} \to T_{r(e)}$ associated to the edges.
Let $A_v := C(T_v)$ for each vertex $v$. For each edge $e$,
 the induced homomorphism $\varphi^*_e : C(T_{r(e)}) \to
C(T_{s(e)})$ determines a $C^*$-correspondence $X_e :=
{_{\varphi^*_e} A_{s(e)}}$ from $A_{r(e)}$ to $A_{s(e)}$. These
correspondences determine a $\Lambda$-system $(A, X, \chi)$.
Quigg \cite{Q} considers a more general situation.

A fourth connection between our construction and the literature
arises in Katsura's realisation of  the Kirchberg algebras using
topological graph $C^*$-algebras. In \cite[section
3]{Katsura2008}, Katsura uses a family of topological graphs
which are fibered over discrete graphs to construct  all nonunital
Kirchberg algebras. We can interpret his construction in terms
of systems of $C^*$-correspondences over $1$-graphs, where each
vertex algebra is $C({\TT})$ and each $C^*$-correspondence is
constructed using two covering maps (see also \cite{De1}). 

\smallskip

Our paper is structured as follows. In
Section~\ref{sec:prelims} we collect basic facts about
$k$-graphs, $C^*$-correspondences, product systems and Fell
bundles, and we establish notation. In
Section~\ref{sec:systems} we define $\Lambda$-systems $(A,X,\chi)$ of
$C^*$-correspondences  and, for systems satisfying
a number of simplifying hypotheses which we refer to
collectively as regularity (see below), the associated
$C^*$-algebras $C^*(A,X,\chi)$. In Section~\ref{sec:theBundle}
we construct from each regular $\Lambda$-system $(A, X, \chi)$
a Fell bundle $E_X$ over the graph groupoid $\Gg_\Lambda$ and
prove that there is an isomorphism $C^*_r(\Gg_\Lambda,
E_X)\cong C^*(A,X,\chi)$. Section~\ref{sec:examples} is devoted
to a discussion of the examples outlined above amongst others.

As already mentioned, for all the $C^*$-algebraic results, we
restrict our attention to the \emph{regular} $\Lambda$-systems
such that the $k$-graph $\Lambda$ is row-finite and has no
sources (that is, for any degree in $\NN^k$ and any vertex $v$
of $\Lambda$, the set of paths with range $v$ and degree $n$ is
finite and nonempty), that the $C^*$-correspondences $X_\gamma$
are full and nondegenerate, and that left action of
 each $A_{r(\gamma)}$ on  $X_\gamma$ is implemented by an
injective homomorphism into the 
compacts. One good
reason for this is that if $\Lambda = T_k$ is the $k$-graph
with one vertex and one path of each degree (that is, $\Lambda$
is isomorphic as a category to $\NN^k$), then $\Lambda$-systems
are precisely the product systems of Hilbert bimodules over
$\NN^k$ considered in \cite{F99, SY}; in
particular, since this special case is not yet well understood,
it seems bootless to worry overmuch about non-regular
$\Lambda$-systems at this juncture.

\medskip

\subsection*{Acknowledgements} The authors wish to express
their gratitude for the financial support as well as the
stimulating atmosphere of the Fields Institute and of the Banff
International Research Station; this work was initiated during
our time at the Fields Institute as participants in the special
session entitled \emph{Structure theory for operator algebras}
in November 2007, and continued in Banff, during the {\em
C*-Algebras Associated to Discrete and Dynamical Systems}
workshop in January 2008. Much of the work was done when the
authors gathered at the University of Wollongong in
August---September 2008: VD and AK would like to thank DP and
AS for their warm hospitality and generous support during this
period.

\section{Preliminaries}\label{sec:prelims}

\subsection{Higher-rank graphs}\label{sec:hrgs}

We will adopt the conventions of \cite{KP, PQR} for $k$-graphs.
Given a nonnegative integer $k$, a \emph{$k$-graph} is a
nonempty countable small category $\Lambda$ equipped with a
functor $d :\Lambda \to \NN^k$ satisfying the
\emph{factorisation property}: for all $\lambda \in \Lambda$
and $m,n \in \NN^k$ such that $d( \lambda )=m+n$ there exist
unique $\mu ,\nu \in \Lambda$ such that $d(\mu)=m$, $d(\nu)=n$,
and $\lambda=\mu \nu$. When $d(\lambda )=n$ we say $\lambda$
has \emph{degree} $n$. We will use $d$ to
denote the degree functor in every $k$-graph in this paper; the
domain of $d$ is always clear from context.

For $k \ge 1$, the standard generators of $\NN^k$ are denoted
$e_1, \dots, e_k$, and for $n \in \NN^k$ and $1 \le i \le k$ we
write $n_i$ for the $i^{\rm th}$ coordinate of $n$.

Given a $k$-graph $\Lambda$, for $n \in \NN^k$, we write $\Lambda^n$ for $d^{-1}(n)$.
The \emph{vertices}
of $\Lambda$ are the elements of $\Lambda^0$. The factorisation
property implies that $o \mapsto \id_o$ is a bijection from the
objects of $\Lambda$ to $\Lambda^0$. We will frequently use
this bijection to silently identify $\Obj(\Lambda)$ with
$\Lambda^0$. The domain and codomain maps in the category
$\Lambda$ therefore become maps $s,r : \Lambda \to \Lambda^0$.
More precisely, for $\alpha \in\Lambda$, the \emph{source}
$s(\alpha)$ is the identity morphism associated with the object
$\dom(\alpha)$ and similarly, $r(\alpha) = \id_{\cod(\alpha)}$.

Note that a $0$-graph is then a countable category whose only
morphisms are the identity morphisms; we think of a $0$-graph as a
collection of isolated vertices.

For $u,v\in\Lambda^0$ and $E \subset \Lambda$, we write $u E$
for $E \cap r^{-1}(u)$ and $E v$ for $E \cap s^{-1}(v)$.
We say that $\Lambda$ is \emph{row-finite} and has \emph{no
sources} if $v\Lambda^n$ is finite and nonempty for all $v \in
\Lambda^0$ and $n \in \NN^k$.

Given $\mu, \nu\in\Lambda$, we say $\lambda$ is a {\em common extension}
of $\mu$ and $\nu$ if $\lambda=\mu\alpha=\nu\beta$ for some $\alpha, \beta\in\Lambda$.
This forces $d(\lambda)\ge d(\mu)\vee d(\nu)$. We say $\lambda$ is {\em minimal} if
$d(\lambda)=d(\mu)\vee d(\nu)$. We define
\[
\Lmin(\mu,\nu)=\{(\alpha, \beta): \mu\alpha=\nu\beta
\; \text{is a minimal common extension of $\mu$ and $\nu$}\}.
\]

Two important examples of $k$-graphs are the following. (1)
For $k \ge 1$ let $\Omega_k$ be the small
category with objects $\mbox{Obj} \,( \Omega_k) = \NN^k$, and
morphisms
$\Omega_k = \{ (m,n) \in \NN^k \times \NN^k : m \le n \}$;
the range and source maps are given by $r ( m , n ) = m$, $s ( m , n ) = n$.
Define $d : \Omega_k\rightarrow \NN^k$  by $d ( m , n ) = n - m$.
Then $\Omega_k$ is a $k$-graph.
(2) Let $T = T_k$ be the semigroup $\NN^k$ viewed as a small
category. If $d : T \rightarrow \NN^k$ is the identity map, then
$( T , d )$ is also a $k$-graph.


\subsection{The path groupoid of a higher rank graph}\label{sec:pathgpd}

In this section we summarise the construction and properties of
the path groupoid associated to a row-finite higher-rank graph
with no sources. For further details, see \cite{KP, PRY2001}.

By a \emph{$k$-graph morphism} from a $k$-graph $\Lambda$ to a
$k$-graph $\Gamma$, we mean a functor $f : \Lambda \to \Gamma$
which respects the degree maps.

\begin{dfn} \label{infpathdef}
Let $\Lambda$ be a row-finite $k$-graph with no sources.  We define the infinite path space of $\Lambda$ by
\[
\Lambda^\infty ~=~ \{ x : \Omega_k  \rightarrow \Lambda  : x \;
\mbox{is a $k$-graph morphism} \}.
\]
\noindent
For each $p \in {\NN}^k$ define $\sigma^p : \Lambda^\infty
\rightarrow \Lambda^\infty$ by $\sigma^p (x) (m,n) = x(m+p,n+p)$
for all $x \in \Lambda^\infty$ and $(m,n) \in \Omega_k$.
\end{dfn}
\noindent We define  $r:\Lambda^\infty\to \Lambda^0$ by $r(x)=x(0,0)$.
Our identification of vertices and objects in $k$-graphs identifies
$x(0,0)$ with $x(0)$, and each $x(n,n)$ with $x(n)$. For $v\in \Lambda^0$, set
\[
v\Lambda^\infty:=\{x\in \Lambda^\infty: r(x)=v\}.
\]
For $\lambda\in\Lambda$ and $z\in s(\lambda)\Lambda^\infty$, there is a unique $x\in\Lambda^\infty$ such that
$x(0,d(\lambda))=\lambda$ and $\sigma^{d(\lambda)}(x)=z$. We denote this infinite path $x$ by
$\lambda z$.
The cylinder sets
\[
Z(\lambda):=\{x\in \Lambda^\infty : x(0,d(\lambda))=\lambda\},
\]
where $\lambda\in \Lambda$ are a basis of compact open sets for a
Hausdorff topology on $\Lambda^\infty$. Note that $v\Lambda^\infty$ is just $Z(v)$.

\begin{dfn} \label{gpdef}
Let $\Lambda$ be a row-finite $k$-graph with no sources. Let
\[
{\mathcal G}_{\Lambda}: =  \{ ( x , n , y ) \in
\Lambda^\infty \times {\ZZ}^k \times \Lambda^\infty :
\sigma^\ell (x) = \sigma^m( y) , n = \ell - m\}.
\]
\noindent Define the range and source maps  $r, s : {\mathcal
G}_\Lambda \rightarrow \Lambda^\infty$ by $r (x , n , y ) = x$,
$s ( x , n , y ) = y$. For  $( x, n, y )$, $( y, \ell , z ) \in
{\mathcal G}_\Lambda$ set $( x , n , y ) (y, \ell , z ): = ( x ,
n + \ell , z )$,  and $( x, n, y )^{-1}: =  ( y , -n , x )$. Then
${\mathcal G}_\Lambda$ is a groupoid called the path groupoid of
$\Lambda$.
\end{dfn}

For $\lambda, \mu\in\Lambda$ with $s(\lambda)=s(\mu)$, we define
\[
Z(\lambda,\mu)=\{(\lambda z, d(\lambda)-d(\mu), \mu z)\in \Gg_\Lambda : z\in s(\lambda)\Lambda^\infty\}.
\]
The family $\Uu_\Lambda:=\{Z(\lambda,\mu):s(\lambda)=s(\mu)\}$ is a basis of compact open bisections
for a
topology under which
$\Gg_\Lambda$ becomes a Hausdorff  \' etale groupoid.
Moreover, for every $g = (x, n, y) \in {\mathcal G}_\Lambda$, the collection of sets
$Z(\lambda,\mu)$ such that there is $z \in \Lambda^\infty$ with $r(z)  = s(\lambda)$
satisfying $x = \lambda z$, $y = \mu z$ and $n = d(\lambda) - d(\mu)$, constitutes
a neighborhood basis for $g$.

Fix $\lambda_1, \lambda_2,\mu_1,\mu_2 \in \Lambda$ with
$s(\lambda_i) = s(\mu_i)$. Suppose that $(x,n,y) \in
Z(\lambda_1,\mu_1) \cap Z(\lambda_2,\mu_2)$. Then $x=\lambda_1z_1=\lambda_2z_2$ and $y=\mu_1z_1=\mu_2z_2$ for some $z_1,z_2\in \Lambda^\infty$.
The factorisation property forces $z_1=\alpha z$ and $z_2=\beta z$ for some
$(\alpha, \beta)\in \Lambda^{min}(\lambda_1,\lambda_2)$. We then have
\begin{equation}\label{eq:path factorisation}
\mu_1\alpha z=\mu_1z_1=y=\mu_2z_2=\mu_2\beta z
\end{equation}
so that $\mu_1\alpha=\mu_2\beta$
is a common extension of $\mu_1$ and $\mu_2$. Moreover that $(x,n,y)\in Z(\lambda_1,\mu_1)\cap Z(\lambda_2,\mu_2)$ forces
$d(\lambda_1) - d(\mu_1) = n = d(\lambda_2) - d(\mu_2)$, so
\begin{align*}
d(\alpha)
    &=(d(\lambda_1)\vee d(\lambda_2))-d(\lambda_1)=
    ((d(\mu_1)+n)\vee (d(\mu_2)+n))-(d(\mu_1)+n)=\\
    &=(d(\mu_1)\vee d(\mu_2))-d(\mu_1),
    \end{align*}
and similarly
\[d(\beta)=(d(\mu_1)\vee d(\mu_2))-d(\mu_2).
\]
In particular $d(\mu_1 \alpha)=d(\mu_1)\vee d(\mu_2)=d(\mu_2\beta)$, and combined with
\eqref{eq:path factorisation} this forces $(\alpha,\beta)\in\Lambda^{min}(\mu_1,\mu_2)$.
 It is easy to check that for any
$(\alpha,\beta) \in \Lmin(\lambda_1, \lambda_2) \cap
\Lmin(\mu_1,\mu_2)$ and any $z \in s(\alpha)\Lambda^\infty$, we
have $(\lambda_1\alpha z, d(\lambda_1) - d(\mu_1), \mu_1\alpha
z) \in Z(\lambda_1,\mu_1) \cap Z(\lambda_2,\mu_2)$. It is
likewise easy to check that if $(\alpha,\beta),
(\alpha',\beta')$ are distinct elements of
$\Lmin(\lambda_1,\lambda_2) \cap \Lmin(\mu_1,\mu_2)$, then
$\Lmin(\lambda_1\alpha, \lambda_2\alpha') = \emptyset$. We
conclude that
\begin{equation}\label{eq:cylinder intersection}
\begin{split}
Z(\lambda_1,\mu_1) \cap Z(\lambda_2,\mu_2)
    &= \bigsqcup_{(\alpha,\beta) \in \Lmin(\lambda_1,\lambda_2) \cap \Lmin(\mu_1,\mu_2)} Z(\lambda_1\alpha, \mu_1\alpha) \\
    &= \bigsqcup_{(\alpha,\beta) \in \Lmin(\lambda_1,\lambda_2) \cap \Lmin(\mu_1,\mu_2)} Z(\lambda_2\beta, \mu_2\beta)
\end{split}
\end{equation}

Let $p := (d(\lambda_1) \vee d(\lambda_2)) - d(\lambda_1)$.
Since
\[
Z(\lambda_1,\mu_1) = \bigsqcup_{\alpha \in s(\lambda_1)\Lambda^p} Z(\lambda_1\alpha,\mu_1\alpha),
\]
we also have
\begin{equation}\label{eq:cylinder setdifference}
\begin{split}
Z(\lambda_1,\mu_1) \setminus Z(\lambda_2,\mu_2)
    = \bigsqcup \{Z(\lambda_1\alpha, \mu_1\alpha) : {}& \alpha \in s(\lambda_1)\Lambda^p,\;
        \nexists \beta \text{ such that } \\
       & (\alpha,\beta) \in \Lmin(\lambda_1,\lambda_2) \cap \Lmin(\mu_1,\mu_2)\}.
\end{split}
\end{equation}

\begin{lem}\label{lem:topology decomp}
Any finite union $\bigcup^m_{i=1} Z(\lambda_i, \mu_i)$ of
elements of $\Uu_\Lambda$ can be expressed as a finite disjoint union
$\bigsqcup_{j=1}^n Z(\sigma_j, \tau_j)$ of elements of $\Uu_\Lambda$ in
such a way that  each $(\sigma_j,\tau_j)$ has the form
 $(\sigma_j,\tau_j)=(\lambda_{i(j)}\nu_j,\mu_{i(j)}\nu_j)$
for some $i(j)\le n$ and $\nu_j\in s(\lambda_{i(j)})\Lambda^0$.
\end{lem}
\begin{proof}
Fix a finite collection of pairs $\{(\lambda_i,\mu_i) : i =
1,\dots,n\}$ such that each $s(\lambda_i) = s(\mu_i)$. Then we
may write $\bigcup^n_{i=1} Z(\lambda_i, \mu_i)$ as the disjoint
union
\[
\bigcup^n_{i=1} Z(\lambda_i, \mu_i)
    = \bigsqcup^n_{i=1} \Big(\bigcap^{i-1}_{j=1} (Z(\lambda_i, \mu_i) \setminus Z(\lambda_j,\mu_j))\Big).
\]
Since intersection distributes over unions, it therefore
suffices to show that given basis sets $Z(\lambda, \mu)$ and
$Z(\sigma, \tau)$ in $\Uu_\Lambda$, the intersection $Z(\lambda, \mu)
\cap Z(\sigma, \tau)$ and the relative complement $Z(\lambda,
\mu) \setminus Z(\sigma, \tau)$ can both be written as finite
disjoint unions of elements of $\Uu_\Lambda$ of the desired form. These statements follow
from \eqref{eq:cylinder intersection}~and~\eqref{eq:cylinder setdifference}.
\end{proof}


\subsection{\texorpdfstring{$C^*$}{C*}-algebras associated to higher-rank
graphs}\label{sec:intro C*}

Given a row-finite $k$-graph $\Lambda$ with no sources,  a
Cuntz-Krieger $\Lambda$-family is a collection $\{t_\lambda :
\lambda \in \Lambda \}$ of partial isometries satisfying the
Cuntz-Krieger relations:
\begin{itemize}
\item $\{t_v : v \in \Lambda^0\}$ is a collection of
    mutually orthogonal projections;
\item $t_\lambda t_\mu = t_{\lambda \mu}$ whenever
    $s(\lambda) = r(\mu)$;
\item $t^*_\lambda t_\lambda = t_{s(\lambda)}$ for all
    $\lambda \in \Lambda$; and
\item $t_v = \sum_{\lambda \in v \Lambda^n} t_\lambda
    t^*_\lambda$ for all $v \in \Lambda^0$ and $n \in
    \NN^k$.
\end{itemize}
In~\cite{RS2005},  a family satisfying only the first three of
these relations is called a Toeplitz-Cuntz-Krieger
$\Lambda$-family. The $k$-graph $C^*$-algebra $C^*(\Lambda)$ is
the universal $C^*$-algebra generated by a Cuntz-Krieger
$\Lambda$-family $\{s_\lambda : \lambda \in \Lambda \}$. That
is, for every Cuntz-Krieger $\Lambda$-family $\{t_\lambda :
\lambda \in \Lambda \}$ there is a homomorphism $\pi_t$ of
$C^*( \Lambda )$ satisfying $\pi_t(s_\lambda) = t_\lambda$ for
all $\lambda \in \Lambda$.

By \cite[Theorem 3.15]{RSY1}, the generators
$s_\lambda$ of  $C^*(\Lambda)$ are
all nonzero.

If $\Lambda$ is a $0$-graph, then it trivially has no sources,
and the last three Cuntz-Krieger relations follow from the
first one. So $C^*(\Lambda)$ is the universal $C^*$-algebra
generated by mutually orthogonal projections $\{s_v : v \in
\Lambda^0\}$; that is $C^*(\Lambda) \cong c_0(\Lambda^0)$.

Let $\Lambda$ be a $k$-graph. A standard argument (see, for example,
\cite[Proposition 2.1]{Raeburn CBMS}) using the universal property
shows that there is a strongly continuous
action $\gamma$ of $\TT^k$ on $C^*(\Lambda)$, called the
\emph{gauge action}, such that $\gamma_z(s_\lambda) =
z^{d(\lambda)} s_\lambda$ for all $z \in \TT^k$ and $\lambda
\in \Lambda$.


\subsection{\texorpdfstring{$C^*$}{C*}-correspondences}\label{sec:correspondences}

We define Hilbert modules following \cite{Lan} and \cite[\S
II.7]{Black}. Let $B$ be a $C^*$-algebra and let $\Hh$ be a
right $B$-module. Then a {\em $B$-valued inner product} on
$\Hh$ is a function $\langle \cdot, \cdot\rangle_B : \Hh \times
\Hh \to B$ satisfying the following conditions for all $\xi,
\eta, \zeta \in \Hh$, $b \in B$ and $\alpha, \beta \in \CC$:
\begin{itemize}
\item $\langle \xi, \alpha\eta + \beta\zeta\rangle_B =
    \alpha\langle \xi, \eta\rangle_B + \beta\langle \xi,
    \zeta\rangle_B$,
\item $\langle \xi, \eta b\rangle_B = \langle \xi,
    \eta\rangle_B\cdot b$,
\item $\langle \xi, \eta\rangle_B = \langle \eta,
    \xi\rangle_B^*$,
\item  $\langle \xi, \xi\rangle_B \ge 0$,  and $\langle \xi,
    \xi\rangle_B = 0$ if and only if $\xi = 0$.
\end{itemize}
If $\Hh$ is complete with respect to the norm  $\| \xi
\|^2: =\| \langle \xi, \xi\rangle_B\|$, then $\Hh$ is said to be a
(right-) {\em Hilbert $B$-module}.  If the range of the inner
product is not contained in any proper ideal in $B$, $\Hh$ is
said to be {\em full}. Note that $B$ may be endowed with the
structure of a full Hilbert $B$-module by taking $\langle \xi,
\eta\rangle_B = \xi^*\eta$ for all $\xi, \eta \in B$.
Let $\Hh_1, \Hh_2$  be Hilbert $B$-modules;
a map $T: \Hh_1 \to \Hh_2$ is an {\em adjointable operator} if there is a
map $T^* : \Hh_2 \to \Hh_1$ such that $\langle T\xi, \eta\rangle_B
= \langle \xi, T^*\eta\rangle_B$ for all $\xi, \eta \in \Hh$.
Such an operator is necessarily linear and bounded. We denote the
collection of all  adjointable operators by $\Ll(\Hh_1, \Hh_2)$.
For $\xi_i \in \Hh_i$ there is a  rank-one adjointable operator
$\theta_{\xi_2, \xi_1} : \Hh_1 \to \Hh_2$ defined by
$\theta_{\xi_2, \xi_1}(\eta) = \xi_2\cdot \langle \xi_1, \eta\rangle_B$.
Note that $\theta_{\xi_2, \xi_1}^* = \theta_{\xi_1, \xi_2}$.
The closure of the span of such operators in $\Ll(\Hh_1, \Hh_2)$ is
denoted $\Kk(\Hh_1, \Hh_2)$ (the space of compact operators).
For a Hilbert $B$-module $\Hh$, both $\Kk(\Hh) = \Kk(\Hh, \Hh)$ and
$\Ll(\Hh) = \Ll(\Hh, \Hh)$  are $C^*$-algebras.   Moreover, $\Kk(\Hh)$ is
an essential ideal in $\Ll(\Hh)$, and  $\Ll(\Hh)$ may be
identified with the multiplier algebra of $\Kk(\Hh)$.
There is a canonical identification  $\Hh \equiv  \Kk(B, \Hh)$
which identifies $\xi \in \Hh$ with the operator $b \mapsto \xi b$.
With this identification we have the factorization
$\theta_{\xi_2, \xi_1} = \xi_2\xi_1^*$. Set $\Hh^*:=\Kk(\Hh,B)$.
We may regard $\Hh^*$ as a left-Hilbert $B$-module.

Let $A$ and $B$ be $C^*$-algebras; then a $C^*$-correspondence
from $A$ to $B$ or more briefly an $A$--$B$
$C^*$-correspondence is a Hilbert $B$-module $\Hh$ together
with a $*$-homomorphism $\phi : A \to \Ll(\Hh)$. We often suppress $\phi$, writing $a
\cdot \xi$ for $\phi(a)\xi$.
A homomorphism $\phi: A \to B$ becomes a homomorphism from
$A$ to $\Kk(B)=B$ when $B$ is regarded as a Hilbert $B$ module as above, and therefore gives
 $B$  the
structure of an $A$--$B$ $C^*$-correspondence
 denoted  ${}_\phi B$.
So it is natural to think of an $A$--$B$ $C^*$-correspondence
as a generalised homomorphism from $A$ to $B$.

A $C^*$-correspondence $\Hh$ is said to be {\em nondegenerate} if
$\overline{\text{span}}\,\{ \phi(a)\xi : a \in A,\; \xi \in \Hh \}=
\Hh$ (some authors have also called such $C^*$-correspondences
\emph{essential}). The above $C^*$-correspondence $_\phi B$ is nondegenerate
if and only if $\phi:A\to B$ is {\em approximately unital}, that is, $\phi$ maps  approximate
units into approximate units (note that by \cite[Theorem II. 7.3.9] {Black}
this condition implies that $\phi$ extends to a unital
map between the multiplier algebras).

As discussed in \cite{Black,EKQR, Landsman:bicategories, Sch},
there is a category $\Cc$ such that $\Obj(\Cc)$ is the class of
$C^*$-algebras, and $\Hom_{\Cc}(A, B)$ consists of all isomorphism
classes of $A$--$B$ $C^*$-correspondences (with identity
morphisms $[A]$). Composition
\[
\Hom_{\Cc}(B, C) \times \Hom_{\Cc}(A, B) \to
\Hom_{\Cc}(A, C)
\]
is defined by $([\Hh_1],[\Hh_2]) \mapsto [\Hh_2 \otimes_B
\Hh_1]$ where $\Hh_2 \otimes_B \Hh_1$ denotes the  balanced tensor
product of $C^*$-correspondences. This
$\Hh_2 \otimes_B \Hh_1$ is called the \emph{internal tensor
product} of $\Hh_2$ and $\Hh_1$ by Blackadar and the
\emph{interior tensor product} by Lance (see
\cite[II.7.4.1]{Black} and \cite[Prop.~4.5]{Lan} and the
following discussion).

Observe that any full right-Hilbert $B$-module $\Hh$ is
a nondegenerate $\Kk(\Hh)$--$B$ $C^*$-correspondence ($\phi$ is the inclusion map);
this is the basic example of an imprimitivity
 bimodule between  $\Kk(\Hh)$ and $B$.
 In this case, $\Kk(\Hh)$ and $B$
are said to be Morita-Rieffel equivalent.
Moreover,  $\Hh^*$ may also be viewed as a $B$--$\Kk(\Hh)$
imprimitivity bimodule.
Given two Hilbert $B$-modules $\Hh_1, \Hh_2$, there is
a natural isomorphism $\Kk(\Hh_2,\Hh_1)\to \Hh_1\otimes_B\Hh_2^*$
such that  $\theta_{\xi_1,\xi_2}\mapsto \xi_1\otimes \xi_2^*$.


\subsection{Representations of
\texorpdfstring{$C^*$}{C*}-correspondences}\label{sec:Pimsner}

Let $\Hh$ be an $A$--$A$ $C^*$-correspondence. Recall from
\cite{Katsura2004a, Pim},  that a representation of $\Hh$ in a $C^*$-algebra $B$
is a pair $(t, \pi)$ where $\pi : A \to B$ is a homomorphism,
$t : \Hh \to B$ is linear, and such that for all $a \in A$ and
$\xi, \eta \in \Hh$, we have $t(a\cdot \xi) = \pi(a)t(\xi)$,
$t(\xi \cdot a) = t(\xi)\pi(a)$, and $\pi(\langle \xi, \eta
\rangle_A) = t(\xi)^* t(\eta)$.

Given a $C^*$-correspondence $\Hh$ over $A$ and a
representation $(t, \pi)$ of $\Hh$ in $B$, there is a
homomorphism $t^{(1)} : \Kk(\Hh) \to B$ satisfying
$t^{(1)}(\theta_{\xi,\eta}) = t(\xi) t(\eta)^*$ for all
$\xi,\eta \in \Hh$ (Pimsner denotes this homomorphism
$\pi^{(1)}$ in \cite{Pim}). The pair $(t, \pi)$ is said to be
\emph{Cuntz-Pimsner covariant} if $t^{(1)} (\phi(a)) = \pi(a)$
for all $a\in \phi^{-1}(\Kk(\Hh))\cap(\ker\phi)^\perp$ (see
\cite[Definition 3.4]{Katsura2004a}).

In the cases of interest later in this paper, $\phi$ is injective and $\phi(A)\subset
\Kk(\Hh)$, so $t^{(1)} \circ \phi$ is a homomorphism from $A$
to $B$. Then the pair $(t, \pi)$ is Cuntz-Pimsner covariant if
$t^{(1)} \circ\phi = \pi$.

Given a $C^*$-correspondence $\Hh$ over $A$, there is a
Cuntz-Pimsner covariant representation $(j_\Hh, j_A)$ in a $C^*$-algebra $\Oo_\Hh$
which is universal in the sense that given another
Cuntz-Pimsner covariant representation $(t, \pi)$ of $\Hh$ in $B$ there is a unique
homomorphism $t \times \pi : \Oo_\Hh \to B$ satisfying $(t
\times \pi) \circ j_\Hh = t$ and $(t \times \pi) \circ j_A =
\pi$. Both $j_\Hh$ and $j_A$ are isometric. Moreover, $\Oo_\Hh$
is unique up to canonical isomorphism.
There is a strongly continuous gauge action
$\gamma:\TT\to \Aut(\Oo_\Hh)$
such that $\gamma_z(j_\Hh(\xi))=zj_\Hh(\xi)$ for $\xi\in\Hh$ and
$\gamma_z(j_A(a))=j_A(a)$ for $a\in A$.


\subsection{Product systems and representations}\label{sec:prodsys}
Let $(P, \cdot)$ be a discrete semigroup with identity $e$ and let $A$ be a $C^*$-algebra.
A {\em product system} of $A$--$A$  $C^*$-correspondences over $P$
is a semigroup $Y=\bigsqcup_{p\in P}Y_p$ such that
\begin{itemize}
\item for each $p\in P$, $Y_p\subset Y$ is a $A$--$A$ $C^*$-correspondence
with inner product $\langle\cdot,\cdot\rangle^p_A$;
\item the identity fiber $Y_e$ is the $A$--$A$ $C^*$-correspondence $_{id}A$;
\item for $p,q\in P\setminus\{e\}$ there is an isomorphism
$\Theta_{p,q}:Y_p\otimes_A Y_q\to Y_{pq}$ satisfying $\Theta_{p,q}(x\otimes_A y)=xy$
for all $x\in Y_p$ and $y\in Y_q$;
\item multiplication in $Y$ by elements of $Y_e=A$ implements the right and left actions of $A$ on each $Y_p$.
\end{itemize}
The homomorphism of $A$ into $\Ll(X_p)$ implementing the
left action on $Y_p$ is denoted  $\phi_p$. The product
system $Y$ is said to be {\em nondegenerate} if each $Y_p$ is a
nondegenerate correspondence. 

Let $B$ be a $C^*$-algebra, and let $Y$ be a nondegenerate product
system such that each $\phi_p$ is an injection into $\Kk(X_p)$.
 A map $\psi:Y\to B$ is called a representation
of $Y$ if, writing $\psi_p$ for $\psi\mid_{Y_p}$, we have
\begin{itemize}
\item each $(\psi_p,\psi_e)$ is a representation of $Y_p$; and
\item $\psi_p(x)\psi_q(y)=\psi_{pq}(xy)$ for all $p,q\in P, x\in Y_p, y\in Y_q$.
\end{itemize}
There is a $C^*$-algebra $\Tt_Y$ (called the Toeplitz algebra)
and a representation $i_Y:Y\to \Tt_Y$ which is universal in the following sense:
$\Tt_Y$ is generated by $i_Y(Y)$ and
for any representation $\psi :Y\to B$ there is a homomorphism $\psi_*:\Tt_Y\to B$ such that $\psi_*\circ i_Y=\psi$.
The representation $i_Y$ is isometric and unique up to isomorphism.

For each $p\in P$ we write
$\psi^{(p)}$ for the homomorphism $(\psi_p)^{(1)}$ discussed in the preceeding section.
 The
representation $\psi$ is {\em Cuntz-Pimsner covariant} if each $(\psi_p, \psi_e)$ is Cuntz-Pimsner covariant, that is if  $\psi^{(p)}\circ\phi_p=\psi_e$ for all $p\in P$.

There is a $C^*$-algebra $\Oo_Y$ and a Cuntz-Pimsner covariant
representation $j_Y:Y\to \Oo_Y$ which is universal in the
following sense: for any Cuntz-Pimsner covariant representation
$\psi :Y\to B$ there is a unique homomorphism $\psi_*:\Oo_Y\to
B$ such that $\psi_*\circ j_Y=\psi$. The pair $(\Oo_Y, j_Y)$ is
unique up to canonical isomorphism and $j_Y$ is isometric. For more details about
product systems, see \cite{F99}.

For $P=\NN^k$, universality allows us to define  strongly continuous gauge actions
$\gamma:\TT^k\to \Aut(\Oo_Y)$ and $\tilde{\gamma}:\TT^k\to \Aut(\Tt_Y)$
such that $\gamma_z(j_Y(y))=z^nj_Y(y)$ and $\tilde{\gamma}_z(i_Y(y))=z^ni_Y(y)$ for $y\in Y_n$.
If $Y$ is nondegenerate and each $\phi_n$ is an injection into $\Kk(Y_n)$, the fixed point algebra $\Oo_Y^\gamma$ is isomorphic to
the inductive limit \[
 \varinjlim_{n \in \NN^k}  \Kk(Y_n).
 \]

We will be mostly interested in nondegenerate product systems over
$P=\NN^k$ (under addition) such that each $\phi_n$ is an injection into $\Kk(Y_n)$.


\subsection{Fell bundles over groupoids} \label{sec:fellgroupoid}

Fell bundles over groupoids were introduced
in~\cite{Yamagami1990} and subsequently studied in~\cite{K1}.
The notion of Fell bundle over a  locally compact groupoid generalizes
both the notion of  $C^*$-algebraic bundle over a  group (see
~\cite[\S 11]{Fell77}) and that of  $C^*$-algebra bundle over a space.
Actions of groupoids on $C^*$-algebra bundles yield Fell bundles
but not all Fell bundles arise in this way.

We assume familiarity with groupoids; we direct the reader to
\cite{Renault1980} or \cite{Paterson1999} for the necessary
background or to \cite{KP} for details on groupoids associated
to higher-rank graphs.

\begin{dfn}[{\cite[Definition~2.1]{K1}}]\label{dfn:Fell bundle}
Let $\Gg$ be a locally compact Hausdorff groupoid and let
$\pi : E \to \Gg$ be a Banach bundle.
Define
\[
E^{(2)} = \{(e_1, e_2) \in E \times E : (\pi(e_1), \pi(e_2))  \in
\Gg^{(2)} \} .
\]
A {\em multiplication} on $E$ is a continuous map $(e_1, e_2)
\mapsto e_1e_2$ from $E^{(2)}$ to $E$ which satisfies:
\begin{enumerate}\renewcommand{\theenumi}{\roman{enumi}}
\item  $\pi(e_1e_2) = \pi(e_1)\pi(e_2)$  for all  $(e_1,
    e_2) \in E^{(2)}$
\item  the induced map $E_{g_1} \times  E_{g_2} \to
    E_{g_1g_2}$ is bilinear for all  $(g_1, g_2) \in
         \Gg^{(2)}$
\item $(e_1e_2)e_3 = e_1(e_2e_3)$  whenever the
    multiplication is defined
\item  $\|e_1e_2\| \le \|e_1\|\,\|e_2\|$ for all $(e_1,
    e_2) \in E^{(2)}$.\saveenumi
\end{enumerate}
An {\em involution} on $E$ is a continuous map $e \mapsto e^*$
from $E$ to $E$ which satisfies:
\begin{enumerate}
\renewcommand{\theenumi}{\roman{enumi}}\restoreenumi
\item  $\pi(e^*) = \pi(e)^{-1}$  for all $e \in E$
\item the restriction of the involution to $E_g$ is conjugate
    linear for all $g \in \Gg$
\item $ e^{**} = e$ for all $ e \in E$.\saveenumi
\end{enumerate}
Finally, the bundle $E$ together with the structure maps is
said to be a {\em Fell bundle} if in addition the following
conditions hold:
\begin{enumerate}\renewcommand{\theenumi}{\roman{enumi}}\restoreenumi
\item $ (e_1e_2)^* = e_2^*e_1^* $ for all $(e_1, e_2) \in
    E^{(2)}$
\item $ \|e^*e\| = \|e\|^2$  for all  $e \in E$
\item for each $e\in E$, $e^*e$ is positive as an element of $E_{s(\pi(e))}$ (which
 is a $C^*$-algebra by (i)-(ix)).
\end{enumerate}
The Fell bundle $E$ is said to be \emph{saturated} if
$E_{g_1}\cdot E_{g_2}$ is total in $E_{g_1g_2}$ for all $(g_1,
g_2) \in \Gg^{(2)}$.
\end{dfn}

\begin{rmk} It follows that $\| e^*\|=\| e\|$ for all $e\in E$.

Note that $E_g$ is a
right-Hilbert $E_{s(g)}$-module with right action implemented by
 multiplication and inner product given by
\[
\langle e_1,e_2\rangle_{E_{s(g)}}=e^*_1e_2.
\]
\noindent
Similarly, $E_g$ can be regarded as a left Hilbert $E_{r(g)}$-module.
Observe that $E$ is saturated if and only if $E_g$ is full as a right Hilbert module for all $g$.
Moreover, $E$ is saturated if and only if
 $E_g$ is an
$E_{r(g)}$--$E_{s(g)}$ imprimitivity bimodule for all $g$.
In this case, if $s(g_1)=x=r(g_2)$, there is an
isomorphism $E_{g_1}\otimes_{E_x}E_{g_2}\cong E_{g_1g_2}$
such that $e_1\otimes e_2\mapsto e_1e_2$.
\end{rmk}

\noindent We denote by $E^{(0)}$ the restriction of the bundle
$E$ to $\Gg^{(0)}$, and observe that $E^{(0)}$ is a
$C^*$-bundle. We denote the corresponding $C^*$-algebra of
sections vanishing at infinity by $C_0 ( \Gg^{(0)}, E^{(0)} )$.

Given a Fell bundle $E$ over an \'{e}tale groupoid $\Gg$ we
construct the $C^*$-algebra $C_r^* ( \Gg , E  )$ as a
completion of $C_c ( \Gg , E )$ in the following way. First we
use the groupoid structure of $\Gg$ to define a product and
involution on $C_c ( \Gg , E )$:
\[
(f_1 f_2)( g ) = \sum_{g = g_1 g_2} f_1 ( g_1 ) f_2 ( g_2 ), \qquad f^* (g) = f ( g^{-1} )^* .
\]

\noindent Then regard $C_c ( \Gg , E  )$
as a pre-Hilbert right $C_0 ( \Gg^{(0)}, E^{(0)} )$-module under pointwise operations,
and form the completion $L^2 ( \Gg , E )$. Left multiplication by
elements of $C_c ( \Gg , E )$ induces an embedding into  $\Ll ( L^2 ( \Gg , E ) )$.
We define $C_r^* (
\Gg , E )$ to be the completion of the image of $C_c ( \Gg , E
)$ in the operator norm. For more details see \cite[\S 3]{K1}.


\section{\texorpdfstring{$\Lambda$}{Lambda}-systems of \texorpdfstring{$C^*$}{C*}-correspondences and representations}\label{sec:systems}

\subsection{$\Lambda$-systems of $C^*$-correspondences}\label{sec:lambda-sys}
Given a $k$-graph $\Lambda$, we define a $\Lambda$-system of
$C^*$-correspondences  by associating a $C^*$-algebra
to each vertex, and a $C^*$-correspondence to each path, as
follows.

\begin{dfn}\label{dfn:Lambda system}
Let $\Lambda$ be a $k$-graph.  Fix
\begin{itemize}
\item for each vertex $v \in \Lambda^0$ a $C^*$-algebra
    $A_v$;
\item for each $\lambda \in \Lambda$ an
    $A_{r(\lambda)}$--$A_{s(\lambda)}$ $C^*$-correspondence
    $X_\lambda$; and
\item for each composable pair $\alpha,\beta$ in $\Lambda$
    a map of  $A_{r(\alpha)}$--$A_{s(\beta)}$ $C^*$-correspondences:\\
    $\chi_{\alpha,\beta} : X_{\alpha}
    \otimes_{A_{s(\alpha)}} X_{\beta} \to X_{\alpha\beta}$
    such that if $\alpha \not\in \Lambda^0$, then
    $\chi_{\alpha,\beta}$ is an isomorphism.
\end{itemize}
Suppose that the $A_v$, the $X_\lambda$ and the
$\chi_{\alpha,\beta}$ have the following properties:
\begin{enumerate}
\item\label{it:Xid} for each $v \in \Lambda^0$, $X_v =
    {}_{id}A_v$ (the identity correspondence
    over $A_v$);
\item\label{it:sys theta} for each $\lambda \in \Lambda$,
    the  maps $\chi_{r(\lambda),\lambda}$ and
    $\chi_{\lambda, s(\lambda)}$ are given by
    \[
    \chi_{r(\lambda),\lambda}(a \otimes_{A_{r(\lambda)}} x) = \phi_\lambda(a)x \qquad\text{and}\qquad
    \chi_{\lambda, s(\lambda)}(x \otimes_{A_{s(\lambda)}} a) = x \cdot a.
    \]
\item\label{it:sys assoc} for each composable triple
    $\alpha,\beta,\gamma \in \Lambda$, the following
    diagram commutes.
\[\begin{CD}
X_{\alpha} \otimes_{A_{s(\alpha)}} X_{\beta} \otimes_{A_{s(\beta)}} X_{\gamma}
 @>\chi_{\alpha,\beta} \otimes \id_{X_{\gamma}}>>
X_{\alpha\beta} \otimes_{A_{s(\beta)}} X_{\gamma} \\
 @V\id_{X_{\alpha}} \otimes \chi_{\beta,\gamma}VV   @V\chi_{\alpha\beta,\gamma}VV \\
X_{\alpha} \otimes_{A_{s(\alpha)}} X_{\beta\gamma}
 @>\hspace{1em}\chi_{\alpha, \beta\gamma}\hspace{1em}>>
X_{\alpha\beta\gamma}
\end{CD}\]
\end{enumerate}
Then we say that \emph{$(A, X, \chi)$ is a $\Lambda$-system of
$C^*$-correspondences}. By the usual abuse of notation, we will
frequently just say that $X$ is a $\Lambda$-system of
$C^*$-correspondences.
\end{dfn}

\begin{dfn}\label{dfn:regular}
We shall say that a system $X$ of $C^*$-correspondences is
\emph{regular} if it satisfies all of the following
assumptions:
\begin{itemize}
\item $\Lambda$ is row-finite and has no sources;
\item each $X_\lambda$ is nondegenerate and full; and
\item each $\phi_\lambda : A_{r(\lambda)} \to
    \Ll(X_\lambda)$ is injective and takes values in
    $\Kk(X_\lambda)$.
\end{itemize}
\end{dfn}

\begin{rmk}
Suppose that $X$ is a regular $\Lambda$-system of
$C^*$-correspondences. Note that each map $\phi_\lambda : A_{r(\lambda)} \to
    \Kk(X_\lambda)$ is approximately unital.
Furthermore, for every $\lambda$, since  $X_\lambda$ is
nondegenerate,   $\chi_{r(\lambda),\lambda} :
X_{r(\lambda)} \otimes_{A_{r(\lambda)}} X_\lambda \to
X_\lambda$ is an isomorphism. Hence by the third bullet-point of
Definition~\ref{dfn:Lambda system},
$\chi_{\alpha,\beta}$ is an isomorphism
for every composable pair $\alpha,\beta$.
\end{rmk}

\begin{rmk}\label{rmk:systems gives functor}
Fix a $\Lambda$-system $X$ of $C^*$-correspondences. Recall
that $\Cc$ denotes the category whose objects are
$C^*$-algebras and whose morphisms are isomorphism classes of
$C^*$-correspondences. There is a contravariant functor $F_X$
from $\Lambda$ to $\Cc$ determined by $F_X(\lambda) =
[X_\lambda]$; in particular, the object map satisfies $F^0_X(v)
= A_v$. As with systems of $k$-morphs (see \cite{KPS2}),
more than one system may determine the same
functor, and
there are functors which cannot be obtained in this way from any system.
\end{rmk}

\begin{rmk}\label{rmk:graph system}
To specify a $\Lambda$-system when $\Lambda$ is a $1$-graph it suffices to give a $C^*$-algebra $A_v$
for each vertex $v\in \Lambda^0$ and a $C^*$-correspondence $X_\lambda$
for each edge $\lambda\in \Lambda^1$. For $\lambda=\lambda_1\cdots\lambda_n\in \Lambda^n$,
with $n\ge 2$ and $\lambda_i\in\Lambda^1$,  define
\[
X_\lambda=X_{\lambda_1}\otimes_{A_{s(\lambda_1)}}X_{\lambda_2}\otimes_{A_{s(\lambda_2)}}\cdots \otimes_{A_{s(\lambda_{n-1})}}X_{\lambda_n};
\]
the maps $\chi_{\alpha, \beta}$ are given by the canonical
isomorphisms.

When $\Lambda$ is a $0$-graph, a $\Lambda$-system of $C^*$-correspondences simply consists of a $C^*$-algebra $A_v$
for each vertex $v\in \Lambda^0$.
\end{rmk}

\begin{examples} \label{ex:systems}
We pause to mention a number of examples from the literature
which can be regarded as $\Lambda$-systems. We will indicate
how each example relates to a $\Lambda$-system, but will
postpone detailed discussions of these and a number of other
examples until Section~\ref{sec:examples}.
\begin{enumerate}\renewcommand{\theenumi}{\roman{enumi}}
\item\label{it:MPT} In \cite[section 3]{MPT}, the authors
    consider two $C^*$-algebras $A$ and $B$, an $A$--$B$
    $C^*$-correspondence $R$, and a $B$--$A$
    $C^*$-correspondence $S$. Suppose $R$ and $S$ are
    nondegenerate with both left actions injective and given by compacts.
    Then they
    prove that the Cuntz-Pimsner
    algebras $\Oo_{R\otimes_B S}$ and $\Oo_{S\otimes_A R}$
    are  Morita-Rieffel equivalent. Let $\Lambda$
    be the path category of the directed graph pictured
    below.
    
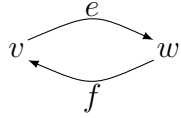
\[\begin{tikzpicture}
    \node[inner sep=1pt] (v) at (-1,0) {$v$};
    \node[inner sep=1pt] (w) at (1,0) {$w$};
    \draw[-latex] (v.north east) .. controls (0,0.5) .. (w.north west) node[pos=0.5, anchor=south, inner sep=1pt] {$e$};
    \draw[-latex] (w.south west) .. controls (0,-0.5) .. (v.south east) node[pos=0.5, anchor=north, inner sep=1pt] {$f$};
    \end{tikzpicture}\]
    If we let $A_v := A$, $A_w := B$, $X_e = R$ and $X_f =
    S$,
    then we obtain a  $\Lambda$-system $X$ of
    correspondences as in Remark~\ref{rmk:graph system}.
    If we assume in addition that $R$ and $S$ are full, then
    the $\Lambda$-system is regular.

\item Recall from \cite{KPS2} that, given $k$-graphs
$\Lambda$ and $\Gamma$, a $\Lambda$-$\Gamma$ $k$-morph $X$ is a
set $X$ together with range and source maps $r : X \to
\Lambda^0$ and $s : X \to \Gamma^0$ and a bijection
\[
\phi : \{(x,\gamma) \in  X \times \Gamma : s(x) = r(\gamma)\}
    \to \{(\lambda,y) \in \Lambda \times X : s(\lambda) = r(x)\}
\]
such that: whenever $\phi(x,\gamma) = (\lambda,y)$, we have
$r(x) = r(\lambda)$, $s(y) = s(\gamma)$, and $d(\lambda) =
d(\gamma)$; and whenever $\phi(x,\gamma) = (\lambda,y)$ and
$\phi(y,\sigma) = (\mu,z)$, we have $\phi(x, \gamma\sigma) =
(\lambda\mu, z)$.
 If $\Gamma$ is an $\ell$-graph,
then a $\Gamma$-system $W$ of $k$-morphs consists, roughly
    speaking, of $k$-graphs $\Lambda_v$ associated to the
    vertices $v \in \Gamma^0$, $k$-morphs $W_\gamma$
    associated to the paths $\gamma \in \Gamma$, and
    compatible isomorphisms $\theta_{\alpha,\beta} :
    W_\alpha \ast_{\Lambda^0_{s(\alpha)}} W_\beta \to
    W_{\alpha\beta}$. Under the technical hypothesis  $(\maltese)$,
   we have by \cite[Proposition~6.4]{KPS2} that
    each $k$-morph $W_\gamma$ gives rise to a full
    nondegenerate $C^*$-correspondence $\Hh(W_\gamma)$
    whose left action is implemented by an injective homomorphism into the compacts
    and it also follows implicitly from the proof of  \cite[Theorem~6.6]{KPS2} that the
    $\theta_{\alpha,\beta}$  determine isomorphisms
    $\chi(\theta_{\alpha,\beta}) : \Hh(W_\alpha)
    \otimes_{C^*(\Lambda_{s(\alpha)})} \Hh(W_\beta) \to
    \Hh(W_{\alpha\beta})$. In particular, if $\Gamma$ is row-finite with no sources,
and each $W_\gamma$ satisfies $(\maltese)$, then the assignments $A_v :=
    C^*(\Lambda_v)$ and $X_\gamma := \Hh(W_\gamma)$ and the
    isomorphisms $\chi(\theta_{\alpha,\beta}) : X_\alpha
    \otimes_{A_{s(\alpha)}} X_\beta \to X_{\alpha\beta}$
    determine a regular $\Gamma$-system of
    $C^*$-correspondences.
\item Every product system of
    $C^*$-correspondences over $\NN^k$  can be regarded as a $T_k$-system
    of $C^*$-correspondences where $T_k$ is the $k$-graph
     isomorphic as
    a category to $\NN^k$ (see section \ref{sec:hrgs}).
\item\label{ex:endo} Given a $C^*$-algebra $A$, let $\End_1(A)$ denote the
    semigroup of  approximately unital endomorphisms of $A$,
    regarded as a category with one object $A$. Let
    $\Lambda$ be a $k$-graph. Then each contravariant
    functor $\varphi : \Lambda \to \End_1(A)$ determines a
    $\Lambda$-system of $C^*$-correspondences.
    Specifically,  $A_v:=A$ for all $v \in \Lambda^0$,
    $X_\lambda := \corresp{\varphi(\lambda)}{A}{}$ for all $\lambda\in \Lambda$, and
    $\chi_{\alpha,\beta}$ is defined by $\chi_{\alpha,\beta}(a
    \otimes b) := \phi_{\beta}(a)b$ for all $a \in X_\alpha
    = \corresp{\varphi(\alpha)}{A}{}$ and $b \in X_\beta =
    \corresp{\varphi(\beta)}{A}{}$.
Conditions
(\ref{it:Xid})~and~(\ref{it:sys theta}) of
Definition~\ref{dfn:Lambda system} are clearly satisfied, and
condition~(\ref{it:sys assoc}) boils down to the identity
\[
\varphi_\gamma(\varphi_\beta(a)b)c = \varphi_{\beta\gamma}(a)\varphi_\gamma(b)c.
\]
The system is regular
    precisely when $\Lambda$ is row-finite with no sources,
    and each $\varphi_\lambda$ is  injective.  This example
  may easily be generalized by replacing the target category
  $\End_1(A)$ with the category of  $C^*$-algebras
   and  approximately unital homomorphisms.
\end{enumerate}
\end{examples}

Fix a $\Lambda$-system $X$ of $C^*$-correspondences. Let $A$
denote the $c_0$ direct sum
\[\textstyle
A = \bigoplus_{v \in \Lambda^0} A_v.
\]
We can regard the $X_\lambda$ as
$A$--$A$ correspondences in the obvious way. 
Hence, for
$n \in \NN^k$ we may define a $C^*$-correspondence $Y_n$ over
$A$ by
\[\textstyle
Y_n = \bigoplus_{\lambda \in \Lambda^n} X_\lambda
\]
(this time, we are taking an $\ell^2$ direct sum). For $\alpha
\in \Lambda$, let $\iota_\alpha : X_\alpha \to Y_{d(\alpha)}$
denote the inclusion map.

By checking that it preserves inner-products, one can see that
for $m,n \in \NN^k\setminus\{0\}$, the formula
\[
\Theta_{m,n}(\iota_\alpha(x) \otimes_A \iota_\beta(y)) :=
\begin{cases}
    \iota_{\alpha,\beta}(\chi_{\alpha,\beta}(x \otimes_{A_{s(\alpha)}} y))
     &\text{ if $s(\alpha) = r(\beta)$}\\
    0_{Y_{m+n}} &\text{ otherwise}
\end{cases}
\]
determines an isomorphism $\Theta_{m,n} : Y_m \otimes_A Y_n \to
Y_{m+n}$.

\begin{prop}\label{prp:assoc prod sys}
With notation as above,
\[
Y =Y_X:= \bigsqcup_{n \in \NN^k}
Y_n
\]
is a product system over $\NN^k$. If $X$ is regular (this
entails, in particular, that $\Lambda$ is row-finite and has no
sources), then $Y_X$ is nondegenerate, and the left action of
$A$ on each  fibre $Y_n$ of $Y_X$ is implemented by an
injection of $A$ into $\Kk(Y_n)$.
\end{prop}
\begin{proof}
Fix $l,m,n \in \NN^k$. Then $Y_l \otimes_A Y_m \otimes_A Y_n$
is spanned by the subspaces
\[
\{\iota_\lambda(X_\lambda) \otimes_{A_{s(\lambda)}} \iota_\mu(X_\mu) \otimes_{A_{s(\mu)}} \iota_\nu(X_\nu)
    : \lambda \in \Lambda^l, \mu \in \Lambda^m, \nu \in \Lambda^n, s(\lambda) = r(\mu), s(\mu) = r(\nu)\}.
\]
Fix $x \in X_\mu, y \in X_\nu$, and $z \in X_\nu$, and for
convenience, write $\bar{x}$ for $\iota_\lambda(x) \in Y_l$ and
similarly for $\mu$ and $\nu$. Then the associativity condition
Definition~\ref{dfn:Lambda system}(\ref{it:sys assoc}) ensures
that
\begin{align*}
(\bar{x}\bar{y})\bar{z}
    &= \theta_{lm,n}(\theta_{l,m} \otimes 1)(\bar{x} \otimes \bar{y} \otimes \bar{z}) \\
    &= \iota_{\lambda\mu\nu}(\chi_{\lambda\mu,\nu}(\chi_{\lambda,\mu} \otimes 1)(x \otimes y \otimes z))
    = \iota_{\lambda\mu\nu}(\chi_{\lambda,\mu\nu}(1 \otimes \chi_{\mu,\nu})(x \otimes y \otimes z))
    = \bar{x}(\bar{y}\bar{z}).
\end{align*}
Hence $Y$ is a product system.

Now suppose that $X$ is regular. It is immediate that $Y$ is
nondegenerate because each $X_\lambda$ is. To see that each
$\phi_n$ is injective, fix $n \in \NN$ and $a \in A \setminus
\{0\}$. Then there is some $v \in \Lambda^0$ such that the
component $a_v$ of $a$ in $A_v$ is nonzero. Since $\Lambda$ has
no sources, there exists $\lambda \in v\Lambda^n$. Since $\phi_{\lambda}$ is injective,
the direct summand $\phi_\lambda(a_v)$ of $\phi_n(a)$ is
nonzero, and hence $\phi_n(a)$ is itself nonzero.  Finally, to
see that $\phi_n$ takes values in $\Kk(Y_n)$, observe that the
$A_v$ span a dense subspace of $A$ and that for a fixed $v$ and
$a \in A_v$, the operator $\phi_n(a) = \oplus_{\lambda \in
v\Lambda^n} \phi_\lambda(a)$ belongs to $\Kk(Y_n)$ because each
$\phi_\lambda(a) \in \Kk(X_\lambda)$ and because $v\Lambda^n$
is finite.
\end{proof}

The construction of the product system $Y$ from the
$\Lambda$-system $X$ is the analogue for systems of
correspondences of the $\Gamma$-bundle construction from
\cite{KPS2}.

Fix a $\Lambda$-system $X$, and paths $\alpha,\beta \in
\Lambda$ with $s(\alpha) = r(\beta)$. As on
\cite[page~42]{Lan}, there is a homomorphism $(\phi_\beta)_* :
\Ll(X_\alpha) \to \Ll(X_\alpha \otimes_{A_{s(\alpha)}}
X_\beta)$ characterised by
\[
(\phi_\beta)_*(T)(x \otimes y) = T(x) \otimes y.
\]
By \cite[Proposition~4.7]{Lan}, if $\phi_\beta(A_{s(\alpha)})
\subset \Kk(X_\beta)$, then
$(\phi_\beta)_*(\Kk(X_\alpha))\subset \Kk(X_\alpha
\otimes_{A_{s(\alpha)}} X_\beta)$, and $(\phi_\beta)_*$ is
injective if $\phi_\beta$ is injective, and surjective if
$\phi_\beta$ is surjective. We define a homomorphism
$i^{\alpha\beta}_{\alpha} : \Ll(X_\alpha) \to
\Ll(X_{\alpha\beta})$ by
\begin{equation}\label{eq:i-maps}
i^{\alpha\beta}_{\alpha}(S) := \chi_{\alpha,\beta} \circ
(\phi_\beta)_*(S) \circ \chi_{\alpha,\beta}^{-1}.
\end{equation}
Hence, if  $\phi_\beta : A_{r(\beta)} \to \Ll(X_\beta)$ takes
values in $\Kk(X_\beta)$, then $i^{\alpha\beta}_\alpha$
restricts to a homomorphism from $\Kk(X_\alpha)$ to
$\Kk(X_{\alpha\beta})$, which is injective if $\phi_\beta$ is.

The maps $i^{\alpha\beta}_{\alpha}$ are compatible
with composition in $\Lambda$ in the sense that for a
composable triple $\alpha,\beta,\gamma$ of $\Lambda$, we have
\[
i^{\alpha\beta\gamma}_{\alpha\beta} \circ i^{\alpha\beta}_{\alpha}
    = i^{\alpha\beta\gamma}_{\alpha}.
\]
To see this, fix $S\in\Ll(X_\alpha)$, apply each of
$(i^{\alpha\beta\gamma}_{\alpha\beta} \circ
i^{\alpha\beta}_{\alpha})(S)$ and $i^{\alpha\beta\gamma}_{\alpha}(S)$
to the image of an elementary tensor $x \otimes y \otimes z$ and
use condition~(3) of Definition~\ref{dfn:Lambda system} to see
that they agree.

Similarly, if $Y$ is a product system over $\NN^k$ and $m,n \in
\NN^k$, let $i_m^{m+n} : \Ll(Y_m) \to  \Ll(Y_{m+n})$ denote the
map obtained as in ~\eqref{eq:i-maps} from the isomorphisms $\Theta_{m,n}:Y_m\otimes_A Y_n\to Y_{m+n}$. If $\phi_n(A) \subset \Kk(Y_n)$ we reuse the
symbol $i_m^{m+n}$ to denote the restriction of $i_m^{m+n}$ to a homomorphism
 from $ \Kk(Y_m)$ to  $\Kk(Y_{m+n})$.

\subsection{Representations of $\Lambda$-systems of $C^*$-correspondences}\label{sec:rep}

\begin{dfn}\label{dfn:representation}
Let $(A,X,\chi)$ be a regular $\Lambda$-system of
$C^*$-correspondences. A \emph{representation} of $X$ in a
$C^*$-algebra $B$ is a pair $(\rho, \pi)$ consisting of
\begin{itemize}
\item linear maps $\rho_{\lambda} : X_\lambda \to B$, and
\item homomorphisms $\pi_v : A_v \to B$
\end{itemize}
which satisfy the following  conditions.
\begin{enumerate}
\item \label{rel:rho=pi} For each $v \in \Lambda^0$,
    $\rho_v = \pi_v$.
\item \label{rel:rho multiplicative}For $\alpha, \beta \in
    \Lambda$, $x \in X_\alpha$ and $y \in X_\beta$,
    \[
        \rho_{\alpha}(x)\rho_{\beta}(y) =
        \begin{cases}
            \rho_{\alpha\beta}(\chi_{\alpha,\beta}(x \otimes_{A_{s(\alpha)}} y)) &\text{ if $s(\alpha) = r(\beta)$} \\
            0_B &\text{ if $s(\alpha) \not= r(\beta)$.}
        \end{cases}
    \]
\item \label{rel:rho and inner-product} For all $\alpha,
    \beta \in \Lambda$ with $d(\alpha) = d(\beta)$, and all
    $x \in X_\alpha$ and $y \in X_\beta$,
    \[
        \rho_\alpha(x)^* \rho_\beta(y) =
        \begin{cases}
            \pi_{s(\alpha)}(\langle x, y \rangle_{A_{s(\alpha)}})  &\text{ if $\alpha = \beta$} \\
            0_B &\text{ otherwise.}
        \end{cases}
    \]
    \setcounter{mycounter}{\value{enumi}}
\end{enumerate}
We say that a representation $(\rho,\pi)$ of $X$ in $B$ is
\emph{Cuntz-Pimsner covariant} if
\begin{enumerate}\setcounter{enumi}{\value{mycounter}}
\item\label{rel:rho,pi covariant} for all $v \in
    \Lambda^0$, all $n \in \NN^k$ and all $a \in A_v$, we
    have
    \[
    \pi_{r(\lambda)}(a) = \sum_{\lambda \in v\Lambda^n} \rho^{(\lambda)}(\phi_\lambda(a)),
    \]
    where $\rho^{(\lambda)}=\rho_\lambda^{(1)}$.
\end{enumerate}
Let $(\rho, \pi)$ be a representation of $(A,X,\chi)$ in a
$C^*$-algebra $B$. We will say that $(\rho,\pi)$ is
\emph{universal} if for any other representation $(\rho',
\pi')$ of $(A,X,\chi)$ in a $C^*$-algebra $C$, there is a
unique homomorphism $\Phi = \Phi_{\rho',\pi'} : B \to C$
satisfying $\Phi \circ \rho_\lambda = \rho'_\lambda$ for all
$\lambda \in \Lambda$, and $\Phi \circ \pi_v = \pi'_v$ for all
$v \in \Lambda^0$. A Cuntz-Pimsner covariant representation
 of $(A,X,\chi)$ is universal if it has the  universal property  described above
 with respect to Cuntz-Pimsner covariant representations $(\rho',\pi')$.
\end{dfn}

\begin{rmk}
Since, in a regular $\Lambda$-system of $C^*$-correspondences, each
$X_\lambda$ is nondegenerate, conditions (\ref{rel:rho
multiplicative})~and~(\ref{rel:rho and inner-product}) of
Definition~\ref{dfn:representation} are then equivalent to the
apparently weaker relations
\begin{itemize}
\item[(A)] $\pi_v(A_v) \perp \pi_w(A_w)$ for distinct $v,w
    \in \Lambda^0$.
\item[(B)] $\rho_\alpha(x)\rho_\beta(y) =
    \rho_{\alpha\beta}(\chi_{\alpha,\beta}(x
    \otimes_{A_{s(\alpha)}} y))$ whenever $s(\alpha) =
    r(\beta)$, $x \in X_\alpha$ and $y \in X_\beta$.
\item[(C)] $\pi_{s(\lambda)} (\langle x, y
    \rangle_{A_{s(\lambda)}}) = \rho_\lambda(x)^*
    \rho_\lambda(y)$ for all $\lambda \in \Lambda$ and $x,y
    \in X_\lambda$.
\end{itemize}
We discuss the complications which would arise in the absence
of the regularity hypothesis at the end of the section in
Remark~\ref{rmk:bollocks}
\end{rmk}


\begin{prop} \label{prp:univ algs same}
Let $\Lambda$ be a row-finite $k$-graph with no sources, and
let $X$ be a regular $\Lambda$-system of $C^*$-correspondences.
Let $Y = Y_X$ be the product system over
$\NN^k$ of $C^*$-correspondences over $A=\oplus A_v$
obtained as above.
\begin{enumerate}
\item\label{it:Y->X} If $\psi$ is a representation of $Y$
    in a $C^*$-algebra $B$, then there is a representation
    $(\rho^\psi, \pi^\psi)$ of $X$ in $B$ given by
    $\rho^\psi_\alpha := \psi \circ \iota_{\alpha}$ and
    $\pi^\psi_v := \psi \circ \iota_v$.
\item\label{it:X->Y} Conversely if $(\rho, \pi)$ is a
    representation of $X$ in a $C^*$-algebra $B$, then
    there is a representation $\psi^{(\rho,\pi)}$ of $Y$ in
    $B$ determined by $\psi^{(\rho,\pi)}(\iota_\alpha(x)) =
    \rho_\alpha(x)$ for $\alpha \in \Lambda$ and $x \in
    X_\alpha$.
\end{enumerate}
These constructions are mutually inverse in the sense that for
a representation $\psi$ of $Y$, we have $\psi^{(\rho^\psi,
\pi^\psi)} = \psi$, and for a representation $(\rho,\psi)$ of
$X$, we have $(\rho^{\psi^{(\rho,\pi)}},
\pi^{\psi^{(\rho,\pi)}}) = (\rho,\pi)$. The representation
 $(\rho^{i_Y}, \pi^{i_Y})$ of $X$ in $\Tt_Y$
is universal in the sense described above.
\end{prop}
\begin{proof}
For~(\ref{it:Y->X}), fix a representation $\psi$ of $Y$ in
$B$, and let $\rho^\psi$ and $\pi^\psi$ be as in ~(\ref{it:Y->X}). We have
$\rho^\psi_v = \pi^\psi_v$ for all $v$ by definition.

Fix $\alpha,\beta \in \Lambda$, $x \in X_\alpha$ and $y \in
X_\beta$. We must establish
Definition~\ref{dfn:representation}(\ref{rel:rho
multiplicative}). When $s(\alpha) = r(\beta)$ this follows
because multiplication in $Y$ is determined by the isomorphisms
$\chi_{\mu,\nu}$, and $\psi$ is multiplicative. When $s(\alpha)
\not= r(\beta)$, the left-hand side of
Definition~\ref{dfn:representation}(\ref{rel:rho
multiplicative}) is equal to zero by the Hewitt-Cohen
factorisation theorem because the $\pi^\psi_v(A_v)$ are
orthogonal.

To verify Definition~\ref{dfn:representation}(\ref{rel:rho and
inner-product}), Fix $\alpha,\beta \in \Lambda$, $x \in
X_\alpha$ and $y \in X_\beta$. Then
\[
\rho^\psi_\alpha(x)^* \rho^\psi_\beta(y)
    = \psi(\iota_\alpha(x))^* \psi(\iota_\beta(y))
    = \psi(\langle \iota_\alpha(x), \iota_\beta(y) \rangle_A),
\]
and the desired relation holds because $\alpha \not= \beta$
forces $\langle \iota_\alpha(x), \iota_\beta(y) \rangle_A = 0$,
and because $\psi$ is a representation. This
establishes~(\ref{it:Y->X}).

For~(\ref{it:X->Y}), let $(\rho,\pi)$ be a representation of
$X$. By definition of the direct sums $Y_n, n\in {\NN}^k$,
there exists a map $\psi^{(\rho,\pi)} : Y \to B$ satisfying the
given formulae, and the restriction $\psi_0$ of $\psi$ to $Y_0
= A$ is a $C^*$-homomorphism. Fix $x,y \in Y$, say $x \in Y_m$
and $y \in Y_n$. We must show that $\psi^{(\rho,\pi)}(xy) =
\psi^{(\rho,\pi)}(x)\psi^{(\rho,\pi)}(y)$. Since multiplication
in $Y$ implements bimodule isomorphisms, it suffices to
consider $x \in X_\mu$ and $y \in Y_\nu$ for some $\mu \in
\Lambda^m$, $\nu \in \Lambda^n$, and show that
\begin{equation}\label{eq:psi(rho,pi) mplicative}
\psi^{(\rho,\pi)}(\iota_\mu(x)\iota_\nu(y)) =
\psi^{(\rho,\pi)}(\iota_\mu(x))\psi^{(\rho,\pi)}(\iota_\nu(y)).
\end{equation}
This follows from
Definition~\ref{dfn:representation}(\ref{rel:rho
multiplicative}), the Hewitt-Cohen factorisation theorem, and
the definition of multiplication in $Y$.
Now fix $x,y \in Y$; we must show that
$\psi^{(\rho,\pi)}(\langle x, y \rangle_A) =
\psi^{(\rho,\pi)}(x)^* \psi^{(\rho,\pi)}(y)$. By
sesqui-linearity and continuity we need only consider $x \in
X_\alpha$ and $y \in X_\beta$ for some $\alpha,\beta \in
\Lambda$. The desired identity then follows from routine
calculations using
Definition~\ref{dfn:representation}(\ref{rel:rho and
inner-product}) and the definition of the inner product in $Y$.
This completes the proof of~\eqref{it:X->Y}.

That the two constructions are inverse to each other follows from the
definitions of the maps involved. For the final statement,
observe that the universal representation $j$ of $Y$ in $\Tt_Y$
determines a generating representation $(\rho^j, \pi^j)$ of $X$
in $\Tt_Y$. If $(\rho,\pi)$ is a representation of $X$ in $B$,
then $(\psi^{(\rho,\pi)})$ is a representation of $Y$ in $B$.
The universal property of $\Tt_Y$ gives a unique homomorphism $\sigma
: \Tt_Y \to B$ such that $\sigma \circ j = \psi^{(\rho,\pi)}$,
and it is  easy to check that this is equivalent to $\sigma \circ \pi^j_v =
\pi_v$ and $\sigma \circ \rho^j_\lambda = \rho_\lambda$ for all
$v \in \Lambda^0$ and $\lambda \in \Lambda$.
\end{proof}

\begin{rmk}
It should also be possible to consider contractive
representations of a $\Lambda$-system of $C^*$-correspondences
using the associated product system as in Proposition
\ref{prp:univ algs same} and thus to formulate a corresponding
dilation theory in the manner of \cite{MS, Skalski-pp08,
Solel2008}.
We  thank the referee for pointing out this connection.
\end{rmk}

Our next goal is to show that the bijection between
representations of a regular $\Lambda$-system  of
$C^*$-correspondences $X$ and representations of the associated
product system $Y_X$ preserves Cuntz-Pimsner covariance.


\begin{prop}\label{prp:preserves CP}
Let $X$ be a regular $\Lambda$-system of $C^*$-correspondences.
Let $Y=Y_X$ be the product system associated to $X$ (see
Proposition \ref{prp:assoc prod sys}).
A representation $\psi$ of $Y$  in a $C^*$-algebra $B$ is Cuntz-Pimsner covariant if and only if the
representation $(\rho^\psi,\pi^\psi)$ of $X$ in $B$ is
Cuntz-Pimsner covariant.
\end{prop}
\begin{proof}
Since $\psi^{(\rho^\psi, \pi^\psi)} = \psi$ for all
representations $(\rho,\pi)$ of $X$ and vice versa, it suffices
to show that if $\psi$ is Cuntz-Pimsner covariant, then
$(\rho^\psi, \pi^\psi)$ has the same property and that if
$(\rho,\pi)$ is Cuntz-Pimsner covariant, then
$\psi^{(\rho,\pi)}$ has the same property.

Before doing this, we establish some notation. For each
$\lambda \in \Lambda$, there is a homomorphism
$\iota^{(\lambda)} : \Kk(X_\lambda) \to \Kk(Y_{d(\lambda)})$ determined by
\begin{equation}\label{eq:iota^(lambda) def}
\iota^{(\lambda)}(\theta_{x,y}) =\theta_{\iota_\lambda(x),
\iota_\lambda(y)}
\end{equation}
for $x,y\in X_\lambda$. For $x,y \in X_\lambda$, we then have
\[\begin{split}
(\rho^{\psi})^{(\lambda)}(\theta_{x,y})
 &= \rho^\psi(x)\rho^\psi(y)^* \\
 &= \psi(\iota_\lambda(x)) \psi(\iota_\lambda(y))^*
 = \psi^{(d(\lambda))}(\theta_{\iota_\lambda(x) ,\iota_\lambda(y)})
 = \psi^{(d(\lambda))}(\iota^{(\lambda)}(\theta_{x,y}));
\end{split}\]
that is,
\begin{equation}\label{eq:compact repn compatibility}
(\rho^\psi)^{(\lambda)} = \psi^{(d(\lambda))} \circ \iota^{(\lambda)} \text{ for all $\lambda \in \Lambda$.}
\end{equation}
Equivalently, for a fixed representation $(\rho,\pi)$ of
$X$,
\begin{equation}\label{eq:compact repn compatibility II}
\rho^{(\lambda)} = (\psi^{(\rho,\pi)})^{(d(\lambda))} \circ \iota^{(\lambda)} \text{ for all $\lambda \in \Lambda$.}
\end{equation}

To prove the equivalence of the two notions of Cuntz-Pimsner
covariance, first note that since $A = \bigoplus_{v \in
\Lambda^0} A_v$ is spanned by the $A_v$ and since, for a
representation $\psi$ of $Y$, we have $\pi^\psi_v =
\psi_0|_{A_v}$ by definition, it will suffice to fix a
representation $\psi$ of $Y$, a vertex $v \in \Lambda^0$, an
element $n \in \NN^k$ and an element $a \in A_v$ and prove that
\[
\sum_{\lambda \in v\Lambda^n} (\rho^\psi)^{(\lambda)}(\phi_\lambda(a))
    = \psi^{(n)}(\phi_n(a)).
\]
We have $\phi_n(a) =
\sum_{\lambda \in v\Lambda^n}
\iota^{(\lambda)}(\phi_\lambda(a))$ by definition of the
product system $Y$, so we may use~\eqref{eq:compact repn
compatibility} to see that
\begin{align*}
\sum_{\lambda \in v\Lambda^n} (\rho^\psi)^{(\lambda)}(\phi_\lambda(a))
    &= \sum_{\lambda \in v\Lambda^n} \psi^{(n)}(\iota^{(\lambda)}(\phi_{\lambda}(a))) \\
    &= \psi^{(n)}\Big(\sum_{\lambda \in v\Lambda^n} \iota^{(\lambda)}(\phi_{\lambda}(a))\Big)
    = \psi^{(n)}(\phi_n(a))
\end{align*}
as required.
\end{proof}
It follows from \cite[Theorem~4.1 and Proposition~5.1]{SY} that
each nondegenerate product system $Y$ in which each $\phi_n$ is
injective with range in $\Kk(Y_n)$ has an isometric universal
Cuntz-Pimsner covariant representation $j_Y : Y\to \Oo_Y$ (it
seems likely that Fowler knew this, but did not make it
explicit in~\cite{F99}).

\begin{cor}\label{cor:universal property}
Let $X$ be a regular $\Lambda$-system of $C^*$-correspondences.
Let $Y=Y_X$ be the product system associated to $X$ (see
Proposition~\ref{prp:assoc prod sys}). The representation
$(\rho^{j_Y}, \pi^{j_Y})$ of $(A,X,\chi)$ in $\Oo_Y$ is
Cuntz-Pimsner covariant. This is a universal Cuntz-Pimsner covariant representation in
the sense that for any other Cuntz-Pimsner covariant
representation $(\rho, \pi)$ of $(A,X,\chi)$ in a $C^*$-algebra
$C$, there is a unique homomorphism $\Psi = \Psi_{\rho,\pi} :
\Oo_Y \to C$ satisfying $\Psi \circ \rho^{j_Y}_\lambda =
\rho_\lambda$ for all $\lambda \in \Lambda$, and $\Psi \circ
\pi^{j_Y}_v = \pi_v$ for all $v \in \Lambda^0$.
\end{cor}

\begin{proof}
The results follow from the universal property of $\Oo_Y$ and
Proposition~\ref{prp:preserves CP}, as in the proof of the final statement of Proposition \ref{prp:univ algs same}.
\end{proof}

\begin{dfn}\label{dfn:C^*(X)}
Let $X$ be a regular $\Lambda$-system of $C^*$-correspondences,
and let $Y_X$ be the product system associated to $X$ as in
Proposition~\ref{prp:assoc prod sys}. We define $C^*(A,X,\chi)
:= \Oo_{Y_X}$, and let $(\rho^X,\pi^X)$ denote the universal
Cuntz-Pimsner covariant representation $(\rho^{j_Y},
\pi^{j_Y})$ of $(A,X,\chi)$ in $C^*(A,X,\chi)$. Given a
Cuntz-Pimsner covariant representation $(\rho,\pi)$ of
$(A,X,\chi)$ in a $C^*$-algebra $C$, we denote the homomorphism induced
 by the universal property
by $\Psi_{\rho,\pi} : C^*(A,X,\chi) \to C$.
\end{dfn}

\subsection{The gauge-invariant uniqueness theorem}\label{sec:giut}
There is a strongly continuous gauge action
$\gamma:\TT^k\to \Aut(C^*(A,X,\chi))$ (inherited from $\Oo_Y$)
such that for all $z\in\TT^k$ we have
\begin{align*}
&&\gamma_z(\pi^X_v(a))&=\pi^X_v(a) &&\text{for all $v\in \Lambda^0$ and $a\in A_v$, }&&\\
&&\gamma_z(\rho^X_\lambda(x))&=z^{d(\lambda)}\rho^X_\lambda(x) &&\text{for all $\lambda\in \Lambda$ and $x\in X_\lambda$},&&
\end{align*}
where $z^{d(\lambda)}$ is multi-index notation: $z^{d(\lambda)}=\prod^k_{i=1}z_i^{d(\lambda)_i}$.
\begin{thm}\label{thm:giut}
Let $X$ be a regular $\Lambda$-system of $C^*$-correspondences.
Let $B$ be a $C^*$-algebra equipped with a strongly continuous
action $\beta:\TT^k\to\Aut(B)$. Let $(\rho, \pi)$ be a
Cuntz-Pimsner covariant representation of $(A,X,\chi)$ in  $B$.
Suppose the following conditions are satisfied:
\begin{enumerate}
\item\label{it:giut-gauge} for all $z\in\TT^k$ we have
\begin{align*}
&&\beta_z(\pi_v(a))&=\pi_v(a) &&\text{for  $v\in \Lambda^0$ and $a\in A_v$, }&&\\
&&\beta_z(\rho_\lambda(x))&=z^{d(\lambda)}\rho_\lambda(x) &&\text{for all $\lambda\in \Lambda$ and $x\in X_\lambda$};&&
\end{align*}
\item\label{it:giut-injective} for all $v\in \Lambda^0,$
    $\pi_v$ is injective.
\end{enumerate}
Then  $\Psi_{\rho,\pi}: C^*(A,X,\chi)\to B$ is injective.
\end{thm}

The proof follows from Proposition \ref{prp:assoc prod sys}, the definition of the gauge action
and the following lemma.

\begin{lem}\label{lem:giut for Y}
Let $Y$ be a nondegenerate product system over $\NN^k$ such that
the left action on each $Y_n$ is injective and by compacts. Let $B$ be a $C^*$-algebra equipped
 with a strongly continuous action $\beta:\TT^k\to\Aut(B)$, and let  $\psi$
be a Cuntz-Pimsner covariant representation of $Y$ in  $B$. Suppose that the following
conditions are satisfied.
\begin{enumerate}
\item For all $z\in\TT^k$, $n \in \NN^k$ and $y \in Y_n$, we have
  $\beta_z(\psi_n(x)) =z^n\psi_n(x)$.
\item The homomorphism $\psi_0: A\to B$ is injective.
\end{enumerate}
Then the induced map $\psi_{*}: \Oo_Y \to B$ is injective.
\end{lem}
\begin{proof}
The map $a\mapsto \int_{\TT^k} \gamma_z(a)\,dz$ is a faithful conditional expectation $\Phi^\gamma:\Oo_Y \to \Oo_Y^\gamma$, and condition (1) ensures that the linear map
 $\Phi^\beta:\psi_*(a)\mapsto \int_{{\mathbb T}^k}\beta_z(\psi_*(a))\,dz$ satisfies $\psi_*\circ \Phi^\gamma=\Phi^\beta\circ\psi_*$.
  It therefore suffices to prove that the restriction
of  $\psi_{*}$ to $\Oo_Y^\gamma$ is injective.
We claim that $\psi_n$ is injective for
all $n\in \NN^k$. Indeed, fix a nonzero $y\in Y_n$. Then $\langle y, y\rangle_A^n\neq 0$,
and since $\langle y, y\rangle_A^n\in Y_0=A$,
we have
\[
\psi_n(y)^*\psi_n(y)=\psi_0(\langle y, y\rangle_A^n)\neq 0
\]
 by (2).
Hence $\psi_n(y)\neq 0$. Moreover, $\psi^{(n)}: \Kk(Y_n) \to  B$ is injective because
$\psi_0$ is (see the opening paragraph of \cite[p.202]{Pim}).

For all $m, n \in \NN^k$
with $m \le n$, we have $j_Y^{(m)}(\Kk(Y_m)) \subset j_Y^{(n)}(\Kk(Y_n))$. These embeddings
are compatible with the natural injection $i_m^n: \Kk(Y_m) \hookrightarrow \Kk(Y_n)$
guaranteed by our assumptions on $Y$. Moreover,
$\bigcup_n j_Y^{(n)}(\Kk(Y_n))$ is a dense subalgebra of $\Oo_Y^\gamma$, and hence
\[
\Oo_Y^\gamma \cong  \varinjlim_{n \in \NN^k}  \Kk(Y_n).
\]
It remains to show that the restriction of  $\psi_{*}$ to each $j_Y^{(n)}(\Kk(Y_n))$ is injective.
But this follows from the fact that each $\psi^{(n)} : \Kk(Y_n) \to  B$ is injective.
\end{proof}

\begin{rmk}\label{rmk:bollocks}
Restricting attention to regular $\Lambda$-systems vastly
simplifies Definition~\ref{dfn:representation} as well as the
proofs of the subsequent results --- namely
Proposition~\ref{prp:univ algs same},
Proposition~\ref{prp:preserves CP},
Corollary~\ref{cor:universal property} and
Theorem~\ref{thm:giut}. We pause to point out how complications
arise in the absence of regularity.

For a start, the results of \cite{F99} and \cite{SY} suggest
that to obtain a nicely-behaved Toeplitz algebra for
non-regular $\Lambda$-systems we would still have to restrict
attention to finitely aligned $k$-graphs $\Lambda$ and to
$\Lambda$-systems which were compactly aligned in the sense
that given compact operators $S$ on $X_\mu$ and $T$ on $X_\nu$
the product $\iota^\lambda_\mu(S)\iota^\lambda_\nu(T)$ is
compact for each minimal common extension of $\mu$ and $\nu$.
This should ensure that $Y_X$ is compactly aligned in the sense
of \cite{F99}. One would then have to add a Nica covariance
relation to those already listed in
Definition~\ref{dfn:representation}: presumably that each
$\rho^{(\mu)}(S)\rho^{(\nu)}(T)$ is equal to the sum of the
$\rho^{(\lambda)}(\iota^\lambda_\mu(S)\iota^\lambda_\nu(T))$
where $\lambda$ ranges over all minimal common extensions of
$\mu$ and $\nu$. Though the technical details of the proof of
Proposition~\ref{prp:univ algs same} would be complicated by
this, it seems clear that the result would generalise.

Since our interest here is in the Cuntz-Pimsner algebra rather
than its Toeplitz extension, the situation is further
complicated in that the issue of an appropriate notion of
Cuntz-Pimsner covariance for compactly aligned
$\Lambda$-systems over finitely aligned $\Lambda$ arises. One
would presumably have to translate the notion of Cuntz-Pimsner
covariance formulated in \cite{SY} into a corresponding notion
for $\Lambda$-systems (the introduction of~\cite{SY} gives an
account of the issues involved). A first guess would be to
impose \cite[Relation~(CP)]{SY} in the product system
associated to each boundary path of $\Lambda$. The situation
here is very unclear and the technical details formidable.
While versions of Proposition~\ref{prp:preserves CP},
Corollary~\ref{cor:universal property} and
Theorem~\ref{thm:giut} should be achievable with an appropriate
definition of Cuntz-Pimsner covariance (forthcoming work of
Carlsen, Larsen, Sims and Vittadello indicates how to address
the gauge-invariant uniqueness theorem), one would not wish to
speculate on the status of later results. For example, the
definition of the $E_x$ at the beginning of
Section~\ref{sec:fibres} is already problematic because the
linking maps  $\iota^{x(0,n)}_{x(0,m)}$ may not be injective
and may not map compacts to compacts; besides which, the
appropriate groupoid is presumably the one described in
\cite{FMY} whose unit space is much more complicated than the
infinite-path space in general.
\end{rmk}


\section{Construction of a Fell
bundle}\label{sec:theBundle}


Fix for this section a row-finite $k$-graph $\Lambda$ with no
sources, and a regular $\Lambda$-system $X$ of
$C^*$-correspondences.

\subsection{The fibres of the Fell bundle}\label{sec:fibres}
\begin{dfn}
Given an infinite path $x \in \Lambda^\infty$, let $E_x$ be the
$C^*$-algebraic inductive limit
\[
E_x := \varinjlim_{n \in \NN^k} \Kk(X_{x(0,n)} )
\]
under the connecting maps $i^{x(0,n)}_{x(0,m)} :
\Kk(X_{x(0,m)}) \to \Kk (X_{x(0,n)})$ defined by
equation~\eqref{eq:i-maps} for $(m,n) \in \Omega_k$. Let
$i^{x}_{x(0,n)} : \Kk(X_{x(0,n)}) \to E_x$ denote the canonical
embedding.
\end{dfn}
%

\begin{prop}\label{prp:Ex as corner}
Fix $x \in \Lambda^\infty$. Let $x^*X$ be the pullback
$\Omega_k$-system given by
\[
(x^*A)_m = A_{x(m)}, \quad (x^*X)_{(m,n)} = X_{x(m,n)}\quad\text{ and }
\quad (x^* \chi)_{(m,n),(n,p)} = \chi_{x(m,n),x(n,p)}.
\]
Let $(\rho, \pi)$ denote the universal representation of $x^*
X$ in $C^*(x^*A, x^*X, x^* \chi)$. Fix $n \in \NN^k$. Then
$\pi_n : A_{x(n)} \to C^*(x^*A, x^*X, x^* \chi)$ extends to
multiplier algebras; denote this extension $\tilde{\pi}_n$. Let
\[
P_n :=
\tilde{\pi}_n(1_{(x^*A)_n}) \in \Mm(C^*(x^*A, x^*X, x^*
\chi)).
\]
Then $P_n$ is full, and $E_{\sigma^n(x)} \cong P_n C^*(x^*A,
x^*X, x^* \chi) P_n$. In particular, $n = 0$ yields
\[
E_x \cong P_0 C^*(x^*A, x^*X, x^* \chi) P_0.
\]
\end{prop}
\begin{proof}
To see that $\pi_n$ extends to multiplier algebras, recall that
since the $X_{x(m,n)}$ are nondegenerate, the product system
$Y$ associated to $x^*X$ by Proposition~\ref{prp:assoc prod
sys} is also nondegenerate. Hence the universal representation
of $Y$ in $\Oo_Y$ extends to a unital homomorphism
$\tilde{j}_Y$ from $\Mm(A) = \Mm\big(\bigoplus_{m \in \NN^k}
A_{x(m)}\big)$ to $\Mm(\Oo_Y)$. The restriction of
$\tilde{j}_Y$ to $\Mm(A_{x(n)})$ is then the desired extension
of $\pi_n$.

To see that $P_n$ is full, let $I(P_n)$ be the ideal of
$C^*(x^*A, x^*X, x^*\chi)$ generated by $P_n$. Since $(\rho, \pi)$
is Cuntz-Pimsner covariant and
$X$ is regular, each $a \in A_{x(0)}$ satisfies $\pi_0(a) =
\rho^{(0,n)}(T)$ for some $T \in \Kk(X_{x(0,n)})$. Since
$\Kk(X_{x(0,n)}) = \clsp\{\theta_{\xi,\eta} : \xi,\eta \in
X_{x(0,n)}\}$ and
\[
\rho^{(0,n)}(\theta_{\xi,\eta}) = \rho_{(0,n)}(\xi) P_n \rho_{(0,n)}(\eta)^*,
\]
it follows that $A_{x(0)} \subset I(P_n)$, and hence that
$I(P_n)$ contains $I(P_0)$. Now for any other $m \in \NN^k$,
and any $a \in A_{x(m)}$, that $X$ is regular, and in
particular that $X_{x(0,m)}$ is full, ensures that $a \in
\clsp\{\langle \xi,\eta \rangle_{A_{x(m)}} : \xi,\eta \in
X_{x(0,m)}\}$, and in particular that
\[
\pi_m(a) \in \clsp\{\rho_{(0,m)}(\xi)^* P_0 \rho_{(0,m)}(\eta) : \xi,\eta \in
X_{x(0,m)}\}.
\]
Thus $\clsp(\bigcup_{n \in \NN^k} \pi_n(A_{x(n)}))$ belongs to
$I(P_n)$. Since each $X_{x(n)}$ is nondegenerate, it follows
that $I(P_n) = C^*(x^*A, x^*X, x^* \chi)$.

For the final statement, we will invoke the universal property
of the direct limit $E_{\sigma^n(x)} = \varinjlim
(\Kk(X_{x(n,n+p)}), i_{x(n,n+p)}^{x(n,n+q)})$. First fix $p \in
\NN^k$. By definition, $\rho^{(n,n+p)}$ is a homomorphism from
$\Kk(X_{x(n,n+p)})$ to $P_n C^*(x^*A, x^*X, x^* \chi) P_n$.
Lemma~3.10 of~\cite{Pim} shows that the $\rho^{(n,n+p)}$ are
compatible with the connecting maps $i_{x(n,n+p)}^{x(n,n+q)}$
in the sense that $\rho^{(n,n+q)} \circ i_{x(n,n+p)}^{x(n,n+q)}
= \rho^{(n, n+p)}$ for all $p$. The universal property of the
direct limit now implies that there is a unique homomorphism
$\rho^{(n,\infty)} : E_{\sigma^n(x)} \to P_n C^*(x^*A, x^*X,
x^* \chi) P_n$ such that $\rho^{(n,\infty)} \circ i^{x(n,
\infty)}_{x(n, n+p)} = \rho^{(n, n+p)}$ for all $p$. Since each
$\pi_{n+p}$ is injective, each $\rho^{(n, n+p)}$ is injective
and hence isometric, which implies that $\rho^{(n,\infty)}$ is
isometric. Finally, since
\begin{align*}
C^*(x^*A,{}& x^*X, x^* \chi) \\
    &= \clsp\{\rho_{(m,q)}(\xi) \rho_{(p,q)}(\eta)^* :  q \in \NN^k, m,p \le q,\;  \xi \in (x^* X)_{(m,q)}, \eta \in (x^* X)_{(p,q)}\},
\end{align*}
we have
\begin{align*}
P_n C^*(x^*A,{}& x^*X, x^* \chi) P_n \\
    &= \clsp\{\rho_{(n,p)}(\xi) \rho_{(n,p)}(\eta)^* : p \ge n,\;\xi,\eta \in X_{x(n,p)}\} \\
    &= \clsp\{\rho^{(n,p)}(T) : p \ge n, T \in \Kk(X_{x(n,p)})\} \\
    &= \rho^{(n,\infty)}(E_x),
\end{align*}
so that $\rho^{(n,\infty)}$ is the desired isomorphism.
\end{proof}

In the Fell bundle we will construct over the $k$-graph
groupoid $\Gg_\Lambda$, the $C^*$-algebras $E_x$ will be the
fibres over the unit space. We now turn to the construction of
the other fibres. We do this using linking algebra techniques
of~\cite{BGR}. Our first step is to associate a $C^*$-algebra
and two complementary full projections in its multiplier
algebra to each groupoid element.

\begin{ntn}
Given paths $\lambda,\mu,\nu \in \Lambda$ such that $s(\lambda)
= s(\mu) = r(\nu)$, there is a homomorphism
$i_{\lambda,\mu}^{\lambda\nu,\mu\nu} : \Ll(X_\lambda \oplus
X_\mu) \to \Ll(X_{\lambda\nu} \oplus X_{\mu\nu})$ obtained in
an analogous manner to the definition of
$i^{\alpha\beta}_\alpha$ in equation~\eqref{eq:i-maps}.
Specifically, regarding $X_\lambda \oplus X_\mu$ as a right
Hilbert $A_{r(\nu)}$ module, \cite[page~42]{Lan} gives a
homomorphism
\[
(\phi_{\nu})_* : \Ll(X_\lambda \oplus X_\mu)
   \to \Ll((X_{\lambda} \oplus X_{\mu}) \otimes_{A_{r(\nu)}} X_\nu).
\]
combining this with the isomorphism $\chi_{\lambda,\nu} \oplus
\chi_{\mu,\nu} : (X_\lambda \oplus X_\mu) \otimes_{A_{r(\nu)}}
X_{\nu} \to X_{\lambda\nu} \oplus X_{\mu\nu}$, we obtain the
desired homomorphism, namely
\[
i_{\lambda,\mu}^{\lambda\nu,\mu\nu}(S) := (\chi_{\lambda,\nu}
\oplus \chi_{\mu,\nu}) \circ (\phi_{\nu})_*(S) \circ (\chi_{\lambda,\nu}
\oplus \chi_{\mu,\nu})^{-1}.
\]

Proposition~4.7 of \cite{Lan} and that $X$ is regular imply that
$i_{\lambda,\mu}^{\lambda\nu,\mu\nu}$ restricts to an injective
approximately unital homomorphism from $\Kk(X_\lambda \oplus X_\mu)$ to
$\Kk(X_{\lambda\nu} \oplus X_{\mu\nu})$.
\end{ntn}

\begin{lem}\label{lem:connecting maps}
Given paths $\lambda,\mu,\nu \in \Lambda$ such that $s(\lambda)
= s(\mu) = r(\nu)$, let $P_{\lambda}, P_{\mu} \in \Ll(X_\lambda
\oplus X_\mu)$ be the projections onto $X_\lambda$ and $X_\mu$
respectively, and define $P_{\lambda\nu}, P_{\mu\nu} \in
\Ll(X_{\lambda\nu} \oplus X_{\mu\nu})$ similarly. Then
\begin{enumerate}
\item\label{it:i(P)} the homomorphism
    $i_{\lambda,\mu}^{\lambda\nu,\mu\nu}$ takes $P_\lambda$
    to $P_{\lambda\nu}$ and $P_\mu$ to $P_{\mu\nu}$;
\item the projections $P_\lambda$ and $P_\mu$ are
    complementary full projections in $\Mm(\Kk(X_\lambda
    \oplus X_\mu)) = \Ll(X_\lambda \oplus X_\mu)$; and
\item\label{it:P<->cpt embedding} under the canonical
    identification of $P_\lambda \Kk(X_\lambda \oplus
    X_\mu) P_\lambda$ with $\Kk(X_\lambda)$, we have
\[
P_\lambda i_{\lambda,\mu}^{\lambda\nu,\mu\nu}(S) P_\lambda =
i^{\lambda\nu}_\lambda(S)
\]
for all $S \in \Kk(X_\lambda)$, and similarly for $\mu$.
\end{enumerate}
\end{lem}
\begin{proof}
A typical spanning element of $X_{\lambda\nu} \oplus
X_{\mu\nu}$ is of the form $(\chi_{\lambda,\nu} \oplus
\chi_{\mu,\nu})((x_\lambda \oplus x_\mu) \otimes x_\nu)$
where $x_\lambda \in X_\lambda$, $x_\mu \in X_\mu$ and $x_\nu
\in X_\nu$. By definition of
$i_{\lambda,\mu}^{\lambda\nu,\mu\nu}$, we have
\[\begin{split}
i_{\lambda,\mu}^{\lambda\nu,\mu\nu}(P_\lambda )\big((\chi_{\lambda,\nu} \oplus
\chi_{\mu,\nu})((x_\lambda \oplus x_\mu) \otimes x_\nu)\big)
&= (\chi_{\lambda,\nu} \oplus \chi_{\mu,\nu})(P_\lambda(x_\lambda \oplus x_\mu) \otimes x_\nu) \\
&= (\chi_{\lambda,\nu}\oplus\chi_{\mu,\nu})(x_\lambda\oplus 0)\otimes x_\nu.
\end{split}\]
Since $(\chi_{\lambda,\nu} \oplus
\chi_{\mu,\nu})((x_\lambda \oplus x_\mu) \otimes x_\nu) =
\chi_{\lambda,\nu}(x_\lambda \otimes x_\nu) \oplus
\chi_{\mu,\nu}(x_\mu \otimes x_\nu)$, we deduce that
\[
i_{\lambda,\mu}^{\lambda\nu,\mu\nu}(P_\lambda)\big((\chi_{\lambda,\nu} \oplus
\chi_{\mu,\nu})((x_\lambda \oplus x_\mu) \otimes x_\nu)\big)
 = P_{\lambda\nu}\big((\chi_{\lambda,\nu} \oplus
\chi_{\mu,\nu})((x_\lambda \oplus x_\mu) \otimes x_\nu)\big).
\]
That $P_\lambda$ and $P_\mu$ are complementary projections is
clear from their definitions. They are full because each of
$X_\lambda$ and $X_\mu$ is a full right-Hilbert
$A_{s(\lambda)}$-module. The last assertion is verified via a
straightforward calculation using spanning elements like the
one above.
\end{proof}

For the next result, we need the following notation. For $g =
(x,n,y) \in \Gg_\Lambda$ we define
\[
D_g : = \{ ( \lambda , \mu ) : x = \lambda z , y = \mu z , d (
\lambda ) - d ( \mu ) = n \} .
\]
The set $D_g$ is ordered by the relation $(\lambda , \mu ) \le (
\lambda' , \mu' )$ if and only if $d(\lambda ) \le d ( \lambda'
)$. Moreover, we  have $(\lambda ,\mu) \le ( \lambda' , \mu'
)$ in $D_g$ if and only if there exists $\nu$ such that $\lambda' =
\lambda \nu$ and $\mu' = \mu \nu$. We may therefore form the
inductive limit
\[
\varinjlim_{( \lambda , \mu ) \in D_g} \Kk ( X_\lambda \oplus X_\mu )
\]
with respect to the connecting maps
$i_{\lambda,\mu}^{\lambda\nu,\mu\nu}$ defined in
Lemma~\ref{lem:connecting maps}. We denote the universal
inclusion maps into the direct limit by
\[
i_{\lambda,\mu}^{g} : \Kk(X_{\lambda} \oplus X_{\mu})
   \to \varinjlim_{(\lambda ,\mu) \in D_g} \Kk ( X_{\lambda} \oplus X_{\mu}) .
\]

\begin{prop} \label{prop:linking}
Fix $g = (x,n,y) \in \Gg_\Lambda$. The maps
$i_{\lambda,\mu}^{g}$ defined above extend to unital
homomorphisms of multiplier algebras such that for each
$(\lambda,\mu) \in D(g)$, the elements $i_{\lambda,\mu}^{g}(P_{\lambda})$
and $i_{\lambda,\mu}^{g}(P_{\mu})$ are complementary full
projections which do not depend on the choice of $(\lambda,\mu)
\in D(g)$. We define $P_x :=
i_{\lambda,\mu}^{g}(P_{\lambda})$ and $P_y :=
i_{\lambda,\mu}^{g}(P_{\mu})$. We then have
\begin{align*}
E_x &\cong P_x \Big(\varinjlim_{(\lambda,\mu)\in D_g} \Kk ( X_{\lambda} \oplus X_{\mu} )\Big) P_x,\text{ and}
 \\
E_y & \cong P_y \Big(\varinjlim_{(\lambda,\mu) \in D_g} \Kk ( X_{\lambda} \oplus X_{\mu} ) \Big) P_y ;
\end{align*}
so $\displaystyle \varinjlim_{(\lambda,\mu) \in D_g}
\Kk(X_{\lambda} \oplus X_{\mu})$ is a linking algebra for $E_x$
and $E_y$.
\end{prop}
\begin{proof}
First observe that the map $i_{\lambda,\mu}^{g}$ is approximately unital for every
$(\lambda,\mu)\in D_g$. This follows because all the linking
 maps
 $i_{\lambda,\mu}^{\lambda\nu,\mu\nu}: \Kk(X_\lambda \oplus X_\mu)\to
\Kk(X_{\lambda\nu} \oplus X_{\mu\nu})$
are approximately unital. Hence the map $i_{\lambda,\mu}^{g}$
extends to a unital map between the multiplier algebras.
 By the preceding lemma, the projections $P_{\lambda}$ and $P_{\mu}$
 are complementary full projections, and since the map $i_{\lambda,\mu}^{g}$
 is approximately unital, their images,
$i_{\lambda,\mu}^{g}(P_{\lambda})$ and
$i_{\lambda,\mu}^{g}(P_{\mu})$, are also complementary full
projections. By the same lemma, $i^g_{\lambda,\mu}(P_\lambda) =
i^g_{\lambda',\mu'}(P_{\lambda'})$ and $i^g_{\lambda,\mu}(P_\mu) =
i^g_{\lambda',\mu'}(P_{\mu'})$ for all $(\lambda,\mu),
(\lambda',\mu') \in D(g)$; therefore $P_x$ and $P_y$ are well defined.
Lemma~\ref{lem:connecting
maps}(\ref{it:P<->cpt embedding}) combined with the universal
property of the direct limit gives the two displayed equations
in the statement of the Proposition. The final statement is a
consequence of the preceding ones.
\end{proof}

\noindent It will frequently be useful to make the
identification:
\begin{equation} \label{eq:offdiag}
 \Kk (X_{\mu}, X_{\lambda}) = P_\lambda  \Kk ( X_{\lambda} \oplus X_{\mu} ) P_\mu.
\end{equation}

\begin{cor} \label{cor:Egdef}
For $g=(x,n,y) \in \Gg_\Lambda$, the $C^*$-algebras $E_x$ and $E_y$ are Morita-Rieffel equivalent with
 imprimitivity bimodule
\[
E_g := P_x \Big(\varinjlim_{(\lambda,\mu) \in D_g}
\Kk(X_{\lambda} \oplus X_{\mu})\Big) P_y .
\]

\noindent  Moreover, under the identification
\eqref{eq:offdiag} we have (as Banach spaces):
\[
E_g = \varinjlim_{(\lambda,\mu) \in D_g} \Kk (X_{\mu}, X_{\lambda} ) .
\]
\end{cor}

\begin{proof}
The first statement follows immediately from Proposition \ref{prop:linking} and \cite[Theorem~1.1]{BGR}.
The second statement follows from the first and \eqref{eq:offdiag}.
 \end{proof}

\begin{rmk}
Let $g,h\in\Gg$ with $s(g)=s(h)=x$. Then
\[
E_{gh^{-1}}\cong E_g \otimes_{E_x} E_{h^{-1}}\cong E_g \otimes_{E_x} E_h^* \cong \Kk ( E_h , E_g ).
\]
\end{rmk}

Using the identification~\eqref{eq:offdiag}, we retain the
notation $i_{\lambda,\mu}^{\lambda\nu,\mu\nu}$ and
$i_{\lambda,\mu}^{g}$ for the embeddings
 \begin{align*}
i_{\lambda,\mu}^{\lambda\nu,\mu\nu} :   \Kk ( X_{\lambda} , X_{\mu} ) &\to
     \Kk ( X_{\lambda\nu} , X_{\mu\nu} ) , \\
     i_{\lambda,\mu}^{g}:    \Kk ( X_{\lambda} , X_{\mu} ) & \to E_g .
\end{align*}
Note that if $\lambda = \mu$ we have $\Kk ( X_{\lambda} ,
X_{\lambda} ) =  \Kk ( X_{\lambda} )$ and our embeddings
satisfy $i_{\lambda, \lambda}^{\lambda\nu, \lambda\nu} =
i_{\lambda}^{\lambda\nu} $. It follows that for all $x \in
\Lambda^\infty$, we may identify $E_{(x, 0, x)}$ with  $E_x$, and
that if $x = \lambda z$ we have $i_{\lambda, \lambda}^{(x, 0,
x)} = i_{\lambda}^{x} $.

\subsection{The operations and topology of the Fell bundle}\label{sec:operations}
The following lemma shows that the connecting maps are
compatible with composition. This will be used later to define
the multiplicative structure of the Fell bundle.

In this lemma and throughout the rest of the paper, given right
Hilbert $A$-modules $X,Y,Z$ and given $T \in \Kk(X,Y)$ and $S
\in \Kk(Y,Z)$, we will use the juxtaposition $ST$ to denote the
composition $S \circ T \in \Kk(X,Z)$.

\begin{lem}\label{lem:preserves composition}
Fix $\lambda_1, \lambda_2,\lambda_3, \mu \in
\Lambda$ such that $s(\lambda_i) = r(\mu)$ for each $i=1,2,3$. Then
for $T_i \in \Kk(X_{\lambda_{i+1}} , X_{\lambda_i} )$,
\[
i_{\lambda_1,\lambda_2}^{\lambda_1\mu,\lambda_2\mu} (T_1)
i_{\lambda_2,\lambda_3}^{\lambda_2\mu,\lambda_3\mu} (T_2) =
i_{\lambda_1,\lambda_3}^{\lambda_1\mu,\lambda_3\mu} (T_1 T_2).
\]
\end{lem}
\begin{proof}
Embed each $\Kk(X_{\lambda_{i+1}}, X_{\lambda_i})$ in
$\Kk(X_{\lambda_1} \oplus X_{\lambda_2} \oplus X_{\lambda_3})$
in a way analogous to~\eqref{eq:offdiag}, and then use that the
linking map
\[
\Kk(X_{\lambda_1} \oplus X_{\lambda_2} \oplus X_{\lambda_3}) \hookrightarrow
 \Kk(X_{\lambda_1\mu} \oplus X_{\lambda_2\mu} \oplus X_{\lambda_3\mu})
\]
is a homomorphism.
\end{proof}

Before setting ourselves up to define the multiplication
between the various $E_g$, we establish a simple technical
lemma that we will use a number of times.

\begin{lem}\label{lem:dense in cartesian}
Fix $n \ge 1$, a composable $n$-tuple $(g_1, \dots, g_n) \in
\Gg_\Lambda^{(n)}$, elements $e_i \in E_{g_i}$ for $i \le n$,
and $\varepsilon > 0$. There exist $x \in \Lambda^\infty$,
$\lambda_1, \dots, \lambda_{n+1} \in \Lambda r(x)$, and $T_i \in
\Kk(X_{\lambda_{i+1}}, X_{\lambda_i})$ for $i \le n$ such that,
for each $i \le n$,
\begin{equation}\label{eq:dense in cartesian}
g_i = (\lambda_i x, d(\lambda_i) - d(\lambda_{i+1}), \lambda_{i+1} x)
    \qquad\text{ and }\qquad
\|i^{g_i}_{\lambda_i, \lambda_{i+1}}(T_i) - e_i\| < \varepsilon.
\end{equation}
\end{lem}
\begin{proof}
We proceed by induction on $n$. The base case $n=1$ follows
from the definition of $E_{g_1}$ as $\varinjlim_{(\lambda_1,
\lambda_{2}) \in D_{g_1}} \Kk(X_{\lambda_{2}},
X_{\lambda_1})$. Now suppose that the result holds for $n \le
k$, and fix $(g_1, \dots, g_{k+1}) \in \Gg_\Lambda^{(k+1)}$ and
$e_i \in E_{g_i}$. Apply the inductive hypothesis with $n = k$
to $(g_1, \dots, g_k)$, $(e_1, \dots, e_k)$, $\varepsilon$ and
with $n = 1$ to $(g_{k+1})$, $(e_{k+1})$, $\varepsilon$ to
obtain $y,z \in \Lambda^\infty$, $\mu_1, \dots, \mu_{k+1} \in
\Lambda r(y)$, $\nu_1, \nu_2 \in \Lambda r(z)$, $R_i \in
\Kk(X_{\mu_{i+1}}, X_{\mu_i})$ and $S \in \Kk(X_{\nu_2},
X_{\nu_1})$ with the appropriate properties.

We have $\mu_{k+1} y = s(g_k) = r(g_{k+1}) = \nu_1 z$, so
\[
\mu' := y(0, (d(\mu_{k+1}) \vee d(\nu_1)) - d(\mu_{k+1})) \quad\text{ and }\quad
\nu' := z(0, (d(\mu_{k+1}) \vee d(\nu_1)) - d(\nu_1))
\]
satisfy $(\mu',\nu') \in \Lmin(\mu_{k+1},\nu_1)$. Let
$\lambda_i := \mu_i\mu'$ for $i \le k+1$, let $\lambda_{k+2} :=
\nu_2\nu'$, and let $x := \sigma^{d(\mu')}(y)$. Then
\[
\lambda_{k+1} = \mu_{k+1}\mu' = \nu_1\nu'\quad\text{ and }\quad
\mu' x = y\text{ and }\nu'x = z.
\]
Hence $g_i = (\mu_i y, d(\mu_i) - d(\mu_{i+1}), \mu_{i+1} y) =
(\lambda_i x, d(\lambda_i) - d(\lambda_{i+1}), \lambda_{i+1}
x)$ for $i \le k$, and similarly $g_{k+1} = (\lambda_{k+1} x,
d(\lambda_{k+1}) - d(\lambda_{k+2}), \lambda_{k+2} x)$. Let
$T_i := i^{\mu_i\mu', \mu_{i+1}\mu'}_{\mu_i, \mu_{i+1}}(R_i)$
for $i \le k$ and let $T_{k+1} := i^{\nu_1\nu',
\nu_2\nu'}_{\nu_1, \nu_2}(S)$. By the compatibility of the
connecting maps in the inductive limit, we have
\[
i^{g_i}_{\lambda_i, \lambda_{i+1}}(T_i)
= \begin{cases}
    i^{g_i}_{\mu_i, \mu_{i+1}}(R_i) &\text{ if $i \le k$} \\
    i^{g_{k+1}}_{\nu_1, \nu_2}(S) &\text{ if $i = k+1$}.
\end{cases}
\]
In particular, each $\|i^{g_i}_{\lambda_i, \lambda_{i+1}}(T_i)
- e_i\| < \varepsilon$ by choice of the $S_i$ and $R$.
\end{proof}


\begin{lem} \label{lem:product}
Fix composable elements $g_1 = (x_1, n_1, x_2)$ and $g_2 =
(x_2, n_2, x_3)$ of $\Gg_\Lambda$. There is a bilinear map
$(e_1, e_2) \mapsto e_1e_2$ from $E_{g_1}\times E_{g_2}$ to
$E_{g_1g_2}$ determined as follows: if
$\lambda_1,\lambda_2,\lambda_3 \in \Lambda$ and $z \in
\Lambda^\infty$ satisfy $n_i = d(\lambda_i) -
d(\lambda_{i+1})$, and $x_i = \lambda_iz$, (so in particular
$g_i = (\lambda_iz, n_i, \lambda_{i+1}z) \in Z(\lambda_i, \lambda_{i+1})$) then for $T_i \in
\Kk(X_{\lambda_{i+1}},X_{\lambda_i})$,
\begin{equation}\label{eq:composition in limit}
i_{\lambda_1,\lambda_2}^{g_1} (T_1)
i_{\lambda_2,\lambda_3}^{g_2} (T_2) =
i_{\lambda_1,\lambda_3}^{g_1g_2} (T_1T_2).
\end{equation}
Moreover, $\|e_1 e_2 \| \le \|e_1\|\|e_2\|$. If
$g_3=(x_3,n_3,x_4)$, so that $g_2$ and $g_3$ are composable and
$e_3 \in E_{g_3}$, then $(e_1e_2)e_3 = e_1(e_2e_3)$.
\end{lem}
\begin{proof}
%
That~\eqref{eq:composition in limit} is bilinear follows from
Lemma~\ref{lem:preserves composition} and the definition of the
direct limit, and we then have
\[
\|i_{\lambda_1,\lambda_3}^{g_1g_2} (T_1T_2)\|
    \le \|i_{\lambda_1,\lambda_2}^{g_1} (T_1)\|\,\|i_{\lambda_2,\lambda_3}^{g_2} (T_2)\|
\]
because the $i_{\lambda_i, \lambda_{i+1}}^{g_i}$ are
restrictions of the injective $C^*$-homomorphisms
\[
i_{\lambda_i, \lambda_{i+1}}^{g_i} : \Kk(X_{\lambda_i} \oplus
X_{\lambda_{i+1}}) \to \varinjlim_{(\mu,\nu) \in D_{g_i}}
\Kk(X_\mu \oplus X_\nu)
\]
and are therefore isometric. It
follows from Lemma~\ref{lem:dense in cartesian} that the
assignment~\eqref{eq:composition in limit} extends uniquely to
the desired bilinear map $(e_1,e_2) \mapsto e_1e_2$, and that
this map satisfies $\|e_1e_2\| \le \|e_1\|\|e_2\|$. For the
final statement, use Lemma~\ref{lem:dense in cartesian} to
approximate $e_3$ by an element of the form
$i^{g_3}_{\lambda_3, \lambda_4}(T_3)$, and then
use~\eqref{eq:composition in limit} and that $(T_1T_2)T_3 =
T_1(T_2T_3)$.
\end{proof}

\noindent We now define an involution on $\prod_{g \in
\Gg_\Lambda} E_g$ which is compatible with the product
structure defined in Lemma \ref{lem:product}. Recall that for
right Hilbert $A$-modules $X, Y$ the adjoint map $T \mapsto
T^*$ defines a conjugate linear isometry from $\Kk ( X, Y)$ to
$\Kk ( Y , X )$.

\begin{lem} \label{lem:star}
For each $g = ( x , n , y) \in \Gg_\Lambda$ there is a
conjugate linear isometry $e \mapsto e^*$ from $E_g$ to
$E_{g^{-1}}$ which is determined by the following property: for
$(\lambda,\mu) \in D_g$, and
$T \in \Kk (X_\mu , X_\lambda )$,
\begin{equation} \label{eq:stardef}
i^g_{\lambda,\mu}(T)^* = i^{g^{-1}}_{\mu,\lambda}(T^*).
\end{equation}
For $e \in E_g$, we have $e^{**} = e$, the element $e^*e$ is
positive in $E_{s(g)}$ and $\| e^* e \| = \| e \|^2$. For $g_1
= ( x_1 , n_1 , x_2 )$, $g_2 = ( x_2 , n_2 , x_3 ) \in
\Gg_\Lambda$, $e_1 \in E_{g_1}$ and $e_2 \in E_{g_2}$ we have
$( e_1 e_2 )^* = e_2^* e_1^*$.
\end{lem}
\begin{proof}
For each $(\lambda,\mu) \in D_g$, the involution on
$\Kk(X_\lambda \oplus X_\mu)$ restricts to the adjoint map from
$\Kk(X_\mu, X_\lambda)$ to $\Kk(X_\lambda, X_\mu)$. By
definition of $C^*$-algebraic direct limits, it follows that
the involution on $\varinjlim_{(\lambda,\mu) \in D_g}
\Kk(X_\lambda \oplus X_\mu)$ restricts to a conjugate linear
map from $E_g$ to $E_{g^{-1}}$ satisfying~\eqref{eq:stardef}.
This map is an isometry because the connecting maps in the
direct limit are all isometric.

To see that $e^{**} = e$, that $e^*e \ge 0$ and that $\|e^*e\|
= \|e\|^2$ for all $e \in E_x$, note that by continuity it
suffices to consider $e = i^g_{\lambda,\mu}(T)$. Two
applications of~\eqref{eq:stardef} then give $e^{**} =
(i^{g^{-1}}_{\mu,\lambda}(T^*))^* = i^g_{\lambda,\mu}(T^{**}) =
e$ because $T^{**} = T$. We have $e^*e =
i^{g^{-1}}_{\mu,\lambda}(T^*)i^g_{\lambda,\mu}(T) =
i^{s(g)}_\mu(T^*T)$ as $i^{s(g)}_\mu$ is a homomorphism. Moreover,
regarding $T$ as an element of $\Kk(X_\lambda \oplus X_\mu)$
and $i^g_{\lambda,\mu}$ as a homomorphism of $\Kk(X_\lambda
\oplus X_\mu)$, we have
\[
\|e^*e\| = \|i^g_{\lambda,\mu}(T^*T)\| = \|i^g_{\lambda,\mu}(T)\|^2 = \|e\|^2.
\]

For the final statement, it suffices by Lemma~\ref{lem:dense in
cartesian} to consider $e_i = i^{g_i}_{\lambda_i,
\lambda_{i+1}}(T_i)$ for $i = 1,2$, and then
\[
(e_1e_2)^*
    = i^{g_1g_2}_{\lambda_1, \lambda_3}(T_1T_2)^*
    = i^{g_2^{-1}g_1^{-1}}_{\lambda_3, \lambda_1}(T_2^* T_1^*)
    = i^{g_2^{-1}}_{\lambda_2, \lambda_3}(T_2^*) i^{g_1^{-1}}_{\lambda_1, \lambda_2}(T_1^*)
    = e^*_2 e^*_1.\qedhere
\]
\end{proof}

The disjoint union $\bigsqcup_{g \in \Gg_\Lambda} E_g$ will
form a Fell bundle over $\Gg_\Lambda$, but we must first endow
it with an appropriate topology. To do this, we first make it
into a Banach bundle by defining a linear space of sections of
continuous norm which is fibrewise dense.

\begin{dfn}\label{dfn:loc.const.sections}
Fix a pair $\lambda,\mu \in \Lambda$ such that $s(\lambda) =
s(\mu)$. Let $T \in \Kk(X_\mu, X_\lambda)$. We define an
element $f^{\lambda,\mu}_T \in \prod_{g \in \Gg_\Lambda} E_g$
by
\[
f^{\lambda,\mu}_T(g) = \begin{cases}
i_{\lambda,\mu}^g(T)&\text{ if $g \in Z(\lambda,\mu)$} \\
0 &\text{ otherwise.}
\end{cases}
\]
\end{dfn}

\begin{lem}\label{lem:sections}
For each $g \in \Gg_\Lambda$, the collection
\[
\{ f^{\lambda,\mu}_T(g) : (\lambda,\mu) \in D_g, T \in \Kk(X_\mu, X_\lambda)\}
\]
is a dense subspace of $E_g$. Moreover, for a fixed pair
$\lambda,\mu \in \Lambda$ with $s(\lambda) = s(\mu)$, the map
$T \mapsto f^{\lambda,\mu}_T$ is a linear map such that
$\|f^{\lambda,\mu}_T(g)\| = \|T\|$ for $g \in Z(\lambda,\mu)$;
in particular, $g \mapsto \|f^{\lambda,\mu}_T(g)\|$ is locally
constant and thus continuous.
\end{lem}
\begin{proof}
The first statement follows immediately from the definitions of
$f^{\lambda,\mu}_T(g)$ and of $E_g$. Fix $\lambda,\mu \in
\Lambda$ with $s(\lambda) = s(\mu)$. The map $T \mapsto
f^{\lambda,\mu}_T$ is linear because the $i^{\lambda,\mu}_g$
are linear, and for $g \in Z(\lambda,\mu)$, we have
$\|f^{\lambda,\mu}_T(g)\| = \|i^g_{\lambda,\mu}(T)\| = \|T\|$
because $i^g_{\lambda,\mu}$ is an injective $C^*$-homomorphism.
The final statement follows because $Z(\lambda,\mu)$ is both
open and closed in $\Gg_\Lambda$.
\end{proof}

We now define
\[
\Ee_X := \lsp\{f^{\lambda,\mu}_T : \lambda,\mu \in \Lambda, s(\lambda) = s(\mu), T \in \Kk(X_\mu, X_\lambda)\}.
\]
This is the collection of sections which will determine the
topology on the Fell bundle.

\begin{prop}\label{prp:Banach-bundle}
Let $E = E_X$ denote the disjoint union $\bigsqcup_{g \in
\Gg_\Lambda} E_g$, and let $\pi : E \to \Gg_\Lambda$ be the
fibre map. There is a unique topology on $E$ under which $\pi :
E \to \Gg_\Lambda$ becomes a Banach bundle and such that the
elements of $\Ee_X$ are continuous.
\end{prop}
\begin{proof}
By \cite[Proposition~1.6]{Fell}, it suffices to check that each
$E_g$ is a Banach space, that $g \mapsto \|f(g)\|$ is
continuous on $\Gg_\Lambda$ for each $f \in \Ee_X$, and that
for each $g \in \Gg$, the set $\{f(g) : f \in \Ee_X\}$ is dense
in $E_g$.   Let $f \in \Ee_X$ be given; to prove that $g \mapsto \|f(g)\|$ is
continuous on $\Gg_\Lambda$, it is sufficient (by  Lemma~\ref{lem:sections}) to observe
that $f$ agrees locally with functions of the form $f^{\lambda,\mu}_T$.
The other statements  follow from the definition of
$E_g$ and Lemma~\ref{lem:sections}.
\end{proof}

\begin{rmk}\label{rmk:basis}
A basis for the topology on $E_X$ obtained from
Proposition~\ref{prp:Banach-bundle}, is described in the second
paragraph of the proof of \cite[Proposition~1.6]{Fell}: it
consists of the sets
\[
W(f,U,\varepsilon) := \{e \in E : \|f(\pi(e)) - e\| < \varepsilon\text{ and } \pi(e) \in U\}
\]
where $U$ varies over open subsets of $\Gg_\Lambda$, $f$ varies
over $\Ee_X$, and $\varepsilon > 0$.

Indeed, we may restrict $U$ to vary over any basis for the
topology on $\Gg_\Lambda$. In particular, we may restrict $U$
to the basis $\Uu_\Lambda$ as defined in subsection
\ref{sec:pathgpd}.
\end{rmk}

The following lemma will be needed for the proof of the main
theorem.

\begin{lem}\label{lem:sleeve}
Fix $f_1, f_2 \in \Ee_X$, an element $g \in \Gg_\Lambda$ and an
open neighbourhood $U$ of $g$. For any $\varepsilon, \delta > 0$ such that
$\|f_1(g) - f_2(g)\| < \varepsilon - \delta$, there exists an
open neighbourhood $V$ of $g$ such that
\[
W(f_2, V, \delta) \subset W(f_1,U,\varepsilon).
\]
\end{lem}
\begin{proof}
By Proposition \ref{prp:Banach-bundle}, each
element $f$ of $\Ee_X$ has the property that $g \mapsto
\|f(g)\|$ is continuous. In particular, since $\|f_1(g) -
f_2(g)\| < \varepsilon - \delta$, there exists a basic
neighbourhood $V$ of $g$ such that for all $h \in V$, we have
$\|f_1(h) - f_2(h)\| < \varepsilon - \delta$. We claim that
this $V$ suffices. Indeed, if $e \in W(f_2, V, \delta)$, then
in particular $\pi(e) \in V \subset U$, and $\|e -
f_1(\pi(e))\| < \varepsilon$ by the triangle inequality.
\end{proof}

\subsection{The main results}\label{sec:main}


\begin{thm} \label{thm:fellout}
Endowed with the multiplication given in
Lemma~\ref{lem:product}, the involution given in
Lemma~\ref{lem:star}, and the topology given in
Proposition~\ref{prp:Banach-bundle}, $E$ forms a saturated Fell
bundle over $\Gg_{\Lambda}$.
\end{thm}
\begin{proof}
To prove that the multiplication is continuous, fix $(g_1, g_2)
\in \Gg_\Lambda^{(2)}$, elements $e_i \in E_{g_i}$, and a basic
neighbourhood $W(f,U,\varepsilon)$ of $e_1e_2$ where $U$
belongs to the basis $\Uu_\Lambda$ described in
Remark~\ref{rmk:basis}. We must find neighbourhoods $W_i =
W(f_i, U_i, \varepsilon_i)$ of the $e_i$ such that
\[
W_1W_2 := \{v_1v_2 : v_i \in W_i, (\pi(v_1), \pi(v_2)) \in \Gg^{(2)}_\Lambda\} \subset W(f,U,\varepsilon).
\]
To do this, we first show that $W(f,U,\varepsilon)$ contains a
smaller neighbourhood of a specific form as justified by the
following claim.

\noindent\textbf{Claim.} There exist $\delta,\eta > 0$ such
that, for any compatible pair of decompositions $g_1 =
(\lambda_1z, d(\lambda_1)-d(\lambda_2),\lambda_2z)$ and $g_2 =
(\lambda_2z, d(\lambda_2) - d(\lambda_3), \lambda_3z)$, and for
any pair of operators $T_i \in \Kk(X_{\lambda_{i+1}},
X_{\lambda_i})$ satisfying
\[
\|i^{g_i}_{\lambda_i, \lambda_{i+1}}(T_i) - e_i\| < \eta,
\]
there exists $n \in \NN^k$ such that $\nu := z(0,n)$ and $V :=
Z(\lambda_1\nu,\lambda_3\nu)$ satisfy $g_1g_2 \in V \subset U$
and
\[
e_1e_2 \in W(f^{\lambda_1,\lambda_3}_{T_1 T_2}, V, \delta) \subset W(f,U,\varepsilon).
\]

To prove the claim, we choose $\delta \in (0, \epsilon/2)$ such
that $\|e_1 e_2 - f(g_1g_2)\| < \varepsilon - 2\delta$. Since
the multiplication map $(e,f) \mapsto ef$ from $E_{g_1} \times
E_{g_2}$ to $E_{g_1g_2}$ satisfies $\|ef\| \le \|e\|\|f\|$,
there exists $\eta > 0$ such that for all $a_i \in E_{g_i}$
with $\|a_i - e_i\| < \eta$, we have $\|a_1a_2 - e_1e_2\| <
\delta$. Fix $T_i \in \Kk(X_{\lambda_{i+1}},
X_{\lambda_i})$ satisfying
\[
\|i^{g_i}_{\lambda_i, \lambda_{i+1}}(T_i) - e_i\| < \eta.
\]
Then setting $a_i := i^{g_i}_{\lambda_i, \lambda_{i+1}}(T_i)$,
our choice of $\eta$ forces
\[
\|f^{\lambda_1,\lambda_3}_{T_1 T_2}(g_1g_2) - f(g_1g_2)\|
    = \|a_1 a_2 - f(g_1g_2)\|
    < \|a_1a_2 - e_1 e_2\| + \|e_1e_2 -  f(g_1g_2)\|
    < \varepsilon - \delta.
\]
By Lemma~\ref{lem:sleeve} applied to $f_1 = f$ and $f_2 =
f^{\lambda_1, \lambda_3}_{T_1T_2}$, there exists a
neighbourhood $V_0$ of $g_1g_2$ such that
\[
e_1e_2 \in W(f^{\lambda_1,\lambda_3}_{T_1 T_2}, V_0, \delta) \subset W(f,U,\varepsilon).
\]
There  exists $n \in \NN^k$ such that, with $\nu := z(0,n)$,
\[
g_1g_2 \in Z(\lambda_1\nu,\lambda_3\nu)
\subset V_0 \cap Z(\lambda_1, \lambda_3)
\]
since such sets form a neighborhood basis for $g$ (see subsection
\ref{sec:pathgpd}). Setting $V := Z(\lambda_1\nu,\lambda_3\nu)$
concludes the proof of the  claim.

Let $\delta$ and $\eta$ be as in the claim. Let
\[
M := \max\{\|e_1\|,\|e_2\|\} + 2\eta,
\]
and define
\[
\kappa := \min\Big\{\frac{\delta}{2M}, \eta\Big\}.
\]
By Lemma~\ref{lem:dense in cartesian}, we may fix $x \in
\Lambda^\infty$ and $\lambda_1, \lambda_2, \lambda_3 \in
\Lambda r(x)$ such that $g_i = (\lambda_i x, d(\lambda_i) -
d(\lambda_{i+1}), \lambda_{i+1} x)$ and $T_i \in
\Kk(X_{\lambda_{i+1}}, X_{\lambda_i})$, and such that
\begin{equation}\label{eq:Ti within kappa}
\|i^{g_i}_{\lambda_i, \lambda_{i+1}}(T_i) - e_i\| < \kappa.
\end{equation}
By the claim, there  exists $n \in \NN^k$ such that, with $\nu := z(0,n)$,
$V = Z(\lambda_1\nu,\lambda_3\nu)$ is a neighbourhood of
$g_1g_2$ contained in $U$ and satisfying
\[
e_1e_2 \in W(f^{\lambda_1,\lambda_3}_{T_1 T_2}, V, \delta) \subset W(f,U,\varepsilon).
\]
Observe that $\|T_i\| < \|e_i\| + \kappa\le\|e_i\| + \eta$, so
$M \ge \|e_i\| + 2\eta > (\|T_i\| - \eta) + 2\eta = \|T_i\| +
\eta$ for each $i$.

Let $V_i := Z(\lambda_i\nu, \lambda_{i+1}\nu)$ for $i = 1,2$,
and observe that $V = V_1V_2$. We claim that the sets
\[
W_i := W\big(f^{\lambda_i, \lambda_{i+1}}_{T_i}, V_i, \kappa\big)
\]
have the desired properties. Indeed, suppose that $b_i \in W_i$
for $i = 1, 2$, and let $h_i := \pi(b_i) \in V_i$. By Lemma
\ref{lem:product} we have
$f^{\lambda_1,\lambda_3}_{T_1T_2}(h_1h_2) =
f^{\lambda_1,\lambda_2}_{T_1}(h_1)f^{\lambda_2,\lambda_3}_{T_2}(h_2)$.
Using this fact and the estimate
\[
\|b_i\|\le\|T_i\|+\kappa\le\|e_i\|+2\kappa\le M,
\]
obtained from the definition of the $b_i$ and~\eqref{eq:Ti
within kappa}, we may calculate:
\begin{align*}
\|b_1 b_2 - f^{\lambda_1,\lambda_3}_{T_1T_2}(h_1h_2)\|
    &\le \|b_1\|\,\|b_2 - f^{\lambda_2,\lambda_3}_{T_2}(h_2)\|
        + \|b_1 - f^{\lambda_1,\lambda_2}_{T_1}(h_1)\|\,\|f^{\lambda_2,\lambda_3}_{T_2}(h_2)\| \\
    &< \kappa\|b_1\| + \kappa\|T_2\| \\
    &\le \delta.
\end{align*}
We have $e_1e_2 \in W_1W_2$ because $\|T_i - e_i\| < \kappa$
forces $e_i \in W_i$ for $i = 1,2$. This completes the proof
that multiplication in $E$ is continuous.

By~\eqref{eq:stardef}, the set $\Ee_X$ is closed under the map
$f \mapsto f^*$ where $f^*(g) := f(g^{-1})^*$. Hence, the
involution on $E_X$ is continuous.

To complete the proof that $\pi : E_X \to \Gg_\Lambda$ is a
Fell bundle, we must verify that $E_X$ is a Banach bundle,
satisfies conditions (i)--(x) of Definition~\ref{dfn:Fell
bundle}, and is saturated. It is a Banach bundle by
Proposition~\ref{prp:Banach-bundle}. It satisfies (i)--(iv) by
Lemma~\ref{lem:product}, and satisfies (v)--(x) by
Lemma~\ref{lem:star}. Finally, that $E$ is saturated follows
from Corollary~\ref{cor:Egdef} and the remark following
Definition~\ref{dfn:Fell bundle}.
\end{proof}

\begin{example}
Let $(A,X,\chi)$ be a regular $\Omega_k$-system of
$C^*$-correspondences. We have $\Omega_k^\infty = \{x_n : n \in
\NN^k\}$ where $x_n$ is the unique infinite path with range
$n$. Note that for each $n$ we have $x_n = \sigma^n(x_0)$. For
each pair $m,n \in \NN^k$ there is a unique element $g_{m,n} =
(x_m, m-n, x_n)$ of $\Gg = \Gg_{\Omega_k}$ whose range is $x_m$
and whose source is $x_n$, so $\Gg$ is isomorphic to the
complete equivalence relation $\NN^k \times \NN^k$.

The pullback $x^*_0 X$ associated to the infinite path with
range $0$ may be identified in the obvious way with $X$ itself.
Hence Proposition~\ref{prp:Ex as corner}, implies that for each
$n$, the fibre $E_{x_n}$ is isomorphic to the corresponding
corner $P_n C^*(A,X,\chi) P_n$. More generally, the fibre
$E_{m,n}$ over the groupoid element $g_{m,n}$ is the subspace
$P_m C^*(A,X,\chi) P_n$.

\end{example}

\begin{rmk} \label{rmk:product}
Fix paths $\lambda_1,\lambda_2,\lambda_3 \in \Lambda$ such that
$s(\lambda_1)=s(\lambda_2)=s(\lambda_3)$, and for $i = 1,2$ let
$T_i \in \Kk(X_{\lambda_{i+1}},X_{\lambda_i})$. It follows from
Lemma \ref{lem:product} that
\begin{equation}\label{eq:product formula}
f^{\lambda_1,\lambda_2}_{T_1}
f^{\lambda_2,\lambda_3}_{T_2} =
f^{\lambda_1,\lambda_3}_{T_1T_2}.
\end{equation}
\end{rmk}
\noindent Recall from \ref{sec:fellgroupoid} that $C_r^* (
\Gg_\Lambda, E_X )$ is defined to be the closure of $C_c (
\Gg_\Lambda , E_X )$ in the operator norm.


\begin{lem} \label{lem:fxdense}
Let $E_X$ be the Fell bundle described in Theorem
\ref{thm:fellout}. Then $\mathcal{E}_X$ is a dense
$*$-subalgebra of $C^*_r  ( \Gg_\Lambda , E_X )$.
\end{lem}
\begin{proof}
Each element of $\Ee_X$ is a linear combination of continuous
sections supported on compact sets of the form $Z(\lambda,
\mu)$. Hence $\Ee_X \subset C_c ( \Gg_\Lambda , E_X )$. As
noted above, equation~\eqref{eq:stardef} implies that $\Ee_X$
is closed under involution. To show that $\Ee_X$ is closed
under convolution, fix $\lambda,\mu,\nu,\tau \in \Lambda$ with
$s(\lambda) = s(\mu)$ and $s(\nu) = s(\tau)$ and fix $T_1 \in
\Kk(X_{\mu}, X_\lambda)$ and $T_2 \in \Kk(X_\tau, X_\nu)$. We
must show that $f^{\lambda,\mu}_{T_1} f^{\nu,\tau}_{T_2} \in
\Ee_X$. To calculate the product, fix $g \in \Gg_\Lambda$.
Since $\Gg_\Lambda$ is \'etale, we have
\[
\big(f^{\lambda,\mu}_{T_1}  f^{\nu,\tau}_{T_2}\big)(g)
    = \sum_{g_1g_2 = g} i^{g_1}_{\lambda,\mu}(T_1) i^{g_2}_{\nu,\tau}(T_2).
\]
Suppose that $g_1,g_2 \in \Gg_\Lambda$ satisfy $g_1g_2 = g$ and
$i^{g_1}_{\lambda,\mu}(T_1) i^{g_2}_{\nu,\tau}(T_2) \not = 0$.
Then $g_1 = (\lambda x, d(\lambda) - d(\mu), \mu x)$ for some
$x \in \Lambda^\infty$, and $g_2 = (\nu y, d(\nu) - d(\tau),
\tau y)$ for some $y \in \Lambda^\infty$. Since $(g_1,g_2) \in
\Gg^{(2)}$, we have
\[
\mu x = s(g_1) = r(g_2) = \nu y,
\]
so there is a unique $(\alpha,\beta) \in \Lmin(\mu,\nu)$ such
that $z := \sigma^{d(\mu) \vee d(\nu)}(\mu x)$ satisfies
\[
\mu \alpha z = \mu x = \nu y = \nu \beta z.
\]
Let $m = d(\lambda) - d(\mu)$ and $n = d(\nu) - d(\tau)$.
Then
\begin{align*}
d(\lambda\alpha) - d(\tau\beta)
    &= d(\lambda) + \big((d(\mu) \vee d(\nu)) - d(\mu)\big) - \big(d(\tau) + \big((d(\mu) \vee d(\nu)) - d(\nu)\big) \big) \\
    &= d(\lambda) - d(\mu) - d(\tau) + d(\nu) \\
    &= m + n.
\end{align*}
We then have $g = g_1g_2 = (\lambda x, n + m, \tau y) =
(\lambda\alpha z, d(\lambda\alpha) - d(\tau\beta), \tau\beta
z)$.

We conclude that if $\big(f^{\lambda,\mu}_{T_1}
f^{\nu,\tau}_{T_2}\big)(g)$ is nonzero, then there is a unique
$(\alpha,\beta) \in \Lmin(\mu,\nu)$ such that
\[
g  = (\lambda\alpha z, d(\lambda\alpha) - d(\tau\beta), \tau\beta z) \in Z(\lambda\alpha,\tau\beta),
\]
and in this case,
\begin{align*}
\big(f^{\lambda,\mu}_{T_1} f^{\nu,\tau}_{T_2}\big)(g)
    &= i^{(\lambda x, m, \mu x)}_{\lambda,\mu}(T_1) i^{(\nu y, n, \tau y)}_{\nu,\tau}(T_2) \\
    &= i_{\lambda\alpha, \mu\alpha}^{(\lambda x, m, \mu x)}\big(i^{\lambda\alpha, \mu\alpha}_{\lambda,\mu}(T_1)\big)
       i_{\nu\beta, \tau\beta}^{(\nu y, n, \tau y)}\big(i^{\nu\beta, \tau\beta}_{\nu,\tau}(T_2)\big) \\
    &= i_{\lambda\alpha, \tau\beta}^g\big(i^{\lambda\alpha, \mu\alpha}_{\lambda,\mu}(T_1) i^{\nu\beta, \tau\beta}_{\nu,\tau}(T_2) \big) \\
    &= f^{\lambda\alpha,\tau\beta}_{i^{\lambda\alpha, \mu\alpha}_{\lambda,\mu}(T_1) i^{\nu\beta, \tau\beta}_{\nu,\tau}(T_2)}(g).
\end{align*}
It follows that
\[
(f^{\lambda,\mu}_{T_1}  f^{\nu,\tau}_{T_2})(g)
    = \begin{cases}
    f^{\lambda\alpha,\tau\beta}_{i^{\lambda\alpha, \mu\alpha}_{\lambda,\mu}(T_1) i^{\nu\beta, \tau\beta}_{\nu,\tau}(T_2)}(g)
        &\text{if $g \in Z(\lambda\alpha,\tau\beta)$ for some $(\alpha,\beta) \in \Lmin(\mu,\nu)$} \\
    0 &\text{otherwise.}
    \end{cases}
\]
Hence
\[
f^{\lambda,\mu}_{T_1}  f^{\nu,\tau}_{T_2}
    = \sum_{(\alpha,\beta) \in \Lmin(\mu,\nu)}
        f^{\lambda\alpha,\tau\beta}_{i^{\lambda\alpha, \mu\alpha}_{\lambda,\mu}(T_1) i^{\nu\beta, \tau\beta}_{\nu,\tau}(T_2)},
\]
which belongs to $\Ee_X$ as required.

It remains to prove that $\Ee_X$ is dense in
$C^*_r(\Gg_\Lambda, E_X)$. Since $C_c( \Gg_\Lambda , E_X)$ is
dense in $C^*_r(\Gg_\Lambda, E_X)$ it suffices to show that we
may approximate each $f \in C_c ( \Gg_\Lambda , E_X )$ by an
element of $\Ee_X$. Given $f \in C_c(\Gg_\Lambda, E_X)$, we may
rewrite $f$ as a finite sum of sections,  each of which is
supported on an element  of $\Uu_\Lambda$. So it suffices to
consider the case where the support of $f$ is contained in $U
\in \Uu_\Lambda$. On such sections, the uniform norm agrees
with the $C^*$-norm. Fix $\varepsilon > 0$. For each point $g
\in U$ fix $f_g \in \Ee_X$ such that $\|f_g(g) - f(g)\| <
\varepsilon/2$. Then there is an element $V_g$ of $\Uu_\Lambda$
containing $g$ such that $\|f_g(h) - f(h)\| < \varepsilon$ for
all $h \in V_g$. We pass to a finite subcover and then
disjointize to obtain a cover of $U$ by disjoint compact open
sets $V_1, \dots, V_n$ and elements $f_1, \dots f_n \in \Ee_X$
such that the support of each $f_i$ is contained in $V_i$ and
each $f_i$ is within $\varepsilon$ of $f$ uniformly on $V_i$.
Now $\sum^n_{i=1} f_i$ approximates $f$ within $\varepsilon$.
\end{proof}

To reduce confusion in the next two results and their proofs, we
adopt the following notation. For $\lambda \in \Lambda$ and
$\xi \in X_\lambda$, there is an element of
$\Kk(A_{s(\lambda)}, X_\lambda)$ defined by $a \mapsto \xi
\cdot a$; we denote this operator by $l_\xi$. In particular,
for $a \in A_v$, we write $l_a$ for the compact operator on
$X_v = A_v$ implemented by left multiplication by $a$. Observe
that $l_\xi^* \in \Kk(X_\lambda, A_{s(\lambda)})$ is defined by
$l_\xi^*(\eta) := \langle \xi, \eta\rangle_{A_{s(\lambda)}}$,
and in particular that $l_\xi l_\eta^* = \theta_{\xi, \eta}$
whilst $l^*_\xi l_\eta =
l_{\langle\xi,\eta\rangle_{A_{s(\lambda)}}}$.

\begin{prop}\label{prp:repn in Cr(E)}
The assignments $\pi_v(a) := f^{v,v}_{l_a}$ and
$\rho_\lambda(x) := f^{\lambda, s(\lambda)}_{l_x}$ determine a
Cuntz-Pimsner covariant representation $(\rho, \pi)$ of
$(A,X,\chi)$ in $C^*_r(\Gg_\Lambda, E_X)$.
\end{prop}
\begin{proof}
Throughout this proof, we use $1_{Z(\lambda,\mu)}$ to denote
the indicator function of a cylinder set in $\Gg_\Lambda$.

We first show that each $\pi_v$ is a $C^*$-homomorphism. They
are linear by definition. We have $1_{Z(v,v)}(g) = 1$ if and
only if $g = g^{-1} = (x,0,x)$ for some $x \in
v\Lambda^\infty$, and since $a \mapsto l_a$ is a homomorphism
we obtain $\pi_v(a)^* = (f^{v,v}_{l_a})^*(g) = 1_{Z(v,v)}(g)
i^{g}_v(l_{a^*}) = f^{v,v}_{a^*}(g) = \pi_v(a^*)$
from~\eqref{eq:stardef}. Similarly,
\[\begin{split}
\pi_v(a) \pi_v(b)
    = (f^{v,v}_{l_a} f^{v,v}_{l_b})(g)
    &= \sum_{g_1g_2 = g} 1_{Z(v,v)}(g_1)1_{Z(v,v)}(g_2) i^{g_1}_v(a)i^{g_2}_v(b) \\
    &= 1_{Z(v,v)}(g) i^g_v(a)i^g_v(b)
    = f^{v,v}_{l_{ab}}(g)
    = \pi_v(ab).
\end{split}\]
since $i^g_v$ and $a \mapsto l_a$ are homomorphisms.

The $\rho_\lambda$ are clearly linear maps, and $\rho_v =
\pi_v$ for $v \in \Lambda^0$ by definition.

We must next check the multiplicative property
Definition~\ref{dfn:representation}(\ref{rel:rho
multiplicative}). Fix $\alpha,\beta \in \Gg_\Lambda$, and fix
$\xi \in X_\alpha$ and $\eta \in X_\beta$. Then for $g \in
\Gg_\Lambda$,
\[
\rho_\alpha(\xi)\rho_\beta(\eta)
    = \sum_{g_1g_2 = g} 1_{Z(\alpha, s(\alpha))}(g_1)1_{Z(\beta,s(\beta))}(g_2)
        i^{g_1}_{\alpha,s(\alpha)}(l_\xi) i^{g_2}_{\beta,s(\beta)}(l_\eta).
\]
The conditions $g_1g_2 = g$, $1_{Z(\alpha, s(\alpha))}(g_1) =
1$ and $1_{Z(\beta,s(\beta))}(g_2) = 1$ combine to force $g =
(\alpha\beta x, d(\alpha\beta), x) \in Z(\alpha\beta,s(\beta))$
for some $x \in s(\beta)\Lambda^\infty$, and $g_1 =
(\alpha\beta x, d(\alpha), \beta x)$ and $g_2 = (\beta x,
d(\beta), x)$; in particular, they force $s(\alpha) =
r(\beta)$, so $\rho_\alpha(\xi)\rho_\beta(\eta) = 0$ if
$s(\alpha) \not= r(\beta)$.

If $s(\alpha) = r(\beta)$, let $g_\alpha := (r(g), d(\alpha),
\sigma^{d(\alpha)}(r(g)))$ and $g_\beta :=
(\sigma^{d(\alpha)}(r(g)), d(\beta), s(g))$. Then we may
continue our calculation:
\begin{align*}
\rho_\alpha(\xi)\rho_\beta(\eta)
    &= 1_{Z(\alpha\beta, s(\beta))}(g)i^{g_\alpha}_{\alpha,s(\alpha)}(l_\xi) i^{g_\beta}_{\beta,s(\beta)}(l_\eta) \\
    &= 1_{Z(\alpha\beta, s(\beta))}(g)i^g_{\alpha\beta,s(\alpha\beta)}\big(i^{\alpha\beta,\beta}_{\alpha, s(\alpha)}(l_\xi) l_\eta\big) \\
    &= 1_{Z(\alpha\beta, s(\beta))}(g)i^g_{\alpha\beta,s(\alpha\beta)}(l_{\chi_{\alpha,\beta}(\xi \otimes \eta)}) \\
    &= \rho_{\alpha\beta}(\chi_{\alpha,\beta}(\xi \otimes \eta)).
\end{align*}

We have next to show that if $d(\alpha) = d(\beta)$, then
$\rho_\alpha(\xi)^* \rho_\beta(\eta) =
\delta_{\alpha,\beta}\pi_{s(\alpha)}(\langle \xi, \eta
\rangle_{A_{s(\alpha)}})$ for all $\xi \in X_\alpha$ and $\eta
\in X_\beta$. Fix such $\alpha$, $\beta$, $\xi$ and $\eta$. For
$g \in \Gg_\Lambda$,
\begin{align*}
\big(\rho_\alpha(\xi)^* \rho_\beta(\eta)\big)(g)
    &= \big((f^{\alpha,s(\alpha)}_\xi)^* f^{\beta,s(\beta)}(\eta)\big)(g) \\
    &= \sum_{g_1g_2 = g} f^{\alpha,s(\alpha)}_\xi(g_1^{-1}) ^* f^{\beta,s(\beta)}_\eta(g_2) \\
    &= \sum_{g_1g_2 = g} 1_{Z(s(\alpha),\alpha)}(g_1)1_{Z(\beta,s(\beta))}(g_2) i^{g_1}_{s(\alpha),\alpha}(l_\xi^*)i^{g_2}_{\beta,s(\beta)}(l_\eta).
\end{align*}
The conditions $g_1g_2 = g$, $ 1_{Z(s(\alpha),\alpha)}(g_1) =
1$ and $1_{Z(\beta,s(\beta))}(g_2) = 1$ combine to force $g =
(x,0,x)$ for some $x \in s(\alpha)\Lambda^\infty$, $g_1 = (x,
-d(\alpha), \alpha x)$ and $g_2 = g^{-1}_1$. In particular, if
$\alpha \not= \beta$, then
$\big(\rho_\alpha(\xi)^* \rho_\beta(\eta)\big)(g) = 0$. Hence,
writing $g_\alpha = (r(g),-d(\alpha), \alpha r(g))$,
\[\begin{split}
\big(\rho_\alpha(\xi)^* \rho_\beta(\eta)\big)(g)
    &= \delta_{\alpha,\beta} 1_{Z(s(\alpha),s(\alpha))}(g) i^{g_\alpha}_{s(\alpha),\alpha}(l_\xi^*) i^{g^{-1}_\alpha}_{\alpha,s(\alpha)}(l_\eta) \\
    &= \delta_{\alpha,\beta} 1_{Z(s(\alpha),s(\alpha))}(g) i^g_{s(\alpha),s(\alpha)}(l_\xi^* l_\eta)
    = \delta_{\alpha,\beta} \pi_{s(\alpha)}(\langle \xi, \eta \rangle_{A_{s(\alpha)}})
\end{split}\]
as required.

It remains to show that $(\rho, \pi)$ is Cuntz-Pimsner
covariant. To do this, we first show that for $\lambda \in
\Lambda$, $T \in \Kk(X_\lambda)$ and $g \in \Gg_\Lambda$, we
have $\rho^{(\lambda)}(T) = 1_{Z(\lambda,\lambda)}(g)
i^g_{\lambda,\lambda}(T)$. By linearity and continuity, it
suffices to consider $T = \theta_{\xi,\eta}$ for some $\xi,\eta
\in X_\lambda$. For this we calculate
\begin{align}
\rho^{(\lambda)}(\theta_{\xi,\eta})
    &= \rho_\lambda(\xi)\rho_\lambda(\eta)^* \nonumber\\
    &= \sum_{g_1g_2 = g} f^{\lambda,s(\lambda)}_\xi(g_1) (f^{\lambda,s(\lambda)}_\eta)^*(g_2) \nonumber\\
    &=\sum_{g_1g_2 = g} f^{\lambda,s(\lambda)}_\xi(g_1) (f^{\lambda,s(\lambda)}_\eta)(g_2^{-1})^* \nonumber\\
    &= \sum_{g_1g_2 = g} 1_{Z(\lambda,s(\lambda))}(g_1) 1_{Z(s(\lambda),\lambda)}(g_2)
        i^{g_1}_{\lambda,s(\lambda)}(l_\xi) i^{g_2}_{s(\lambda),\lambda}(l^*_\eta)\label{eq:rho(theta)}
\end{align}
by~\eqref{eq:stardef}. The conditions $g_1g_2 = g$,
$1_{Z(\lambda,s(\lambda))}(g_1) = 1$ and
$1_{Z(s(\lambda),\lambda)}(g_2) = 1$ combine to force $g =
(\lambda x,0,\lambda x)$ for some $x \in
s(\lambda)\Lambda^\infty$, $g_1 = (\lambda x, d(\lambda), x)$
and $g_2 = g^{-1}_1$. Hence~\eqref{eq:rho(theta)} together with
the characterisation~\eqref{eq:composition in limit} of
multiplication in $E_X$ implies that
\[
\rho^{(\lambda)}(\theta_{\xi,\eta}) = 1_{Z(\lambda,\lambda)}(g)
i^g_{\lambda,\lambda}(\theta_{\xi,\eta})
\]
as claimed.

By Condition~\ref{dfn:Lambda system}(\ref{it:sys theta}), we
have $\phi_\lambda(a) = i^\lambda_{r(\lambda)}(l_a)$ for all $a
\in A_{r(\lambda)}$. Hence
\begin{equation}\label{eq:rho->i}
 \rho^{(\lambda)}(\phi_\lambda(a))
    = 1_{Z(\lambda,\lambda)}(g) i^g_{\lambda,\lambda}(i^{\lambda}_{r(\lambda)}(l_a))
    = 1_{Z(\lambda,\lambda)}(g) i^g_{r(\lambda),r(\lambda)}(l_a).
\end{equation}

We are now ready to establish that $(\rho, \pi)$ is
Cuntz-Pimsner covariant. Indeed, fix $n \in \NN^k$, $v \in
\Lambda^0$ and $a \in A_v$. Observe that $Z(v,v) =
\bigsqcup_{\lambda \in v\Lambda^n} Z(\lambda,\lambda)$. Hence,
for $g \in \Gg_\lambda$, using equation~\eqref{eq:rho->i} to
obtain the first equality, we calculate
\[
\sum_{\lambda \in v\Lambda^n} \rho^{(\lambda)}(\phi_\lambda(a))
    = \sum_{\lambda \in v\Lambda^n} 1_{Z(\lambda,\lambda)}(g) i^g_{r(\lambda),r(\lambda)}(l_a)
    = 1_{Z(v,v)}(g) i^g_{r(\lambda),r(\lambda)}(l_a)
    = f^{v,v}_{l_a}(g)
    = \pi_v(a),
\]
so $(\rho, \pi)$ is Cuntz-Pimsner covariant as required.
\end{proof}

Our second main theorem identifies the cross-sectional algebra
of the Fell bundle with the $C^*$-algebra $C^*(A,X,\chi)$ of
the $\Lambda$-system $X$ (see Section~\ref{sec:systems}).


\begin{thm}\label{thm:isomorphism}
There is an isomorphism $\Psi : C^*(A,X,\chi) \to C^*_r(\Gg_\Lambda, E_X)$
such that under the canonical identifications $A_v = \Kk(X_v,
X_v)$ and $X_\lambda = \Kk(X_{s(\lambda)}, X_\lambda)$, we have
\begin{align*}
\Psi(\pi^X_v(a)) &= f^{v,v}_{l_a}\quad\text{ for all $v\in \Lambda^0$ and $a \in A_v$, and}\\
\Psi(\rho^X_\lambda(x)) &= f^{\lambda, s(\lambda)}_{l_x} \quad\text{ for all $\lambda \in \Lambda$ and $x \in X_\lambda$.}
\end{align*}
\end{thm}
\begin{proof}
Let $(\rho,\pi)$ be the Cuntz-Pimsner covariant representation
of $(A,X,\chi)$ in $C^*_r(\Gg_\Lambda, E_X)$ obtained from
Proposition~\ref{prp:repn in Cr(E)}. This gives a homomorphism
\[
\Psi = \Psi_{\rho, \pi} : C^*(A,X,\chi) \to C^*_r(\Gg_\Lambda, E_X)
\]
as in Definition \ref{dfn:C^*(X)} satisfying the given formulae. It remains to
show that $\Psi$ is bijective.

We use Theorem \ref{thm:giut} to prove injectivity. Observe
that for each $v\in \Lambda^0$, the map $\pi_v$ is injective,
so condition Theorem \ref{thm:giut}(\ref{it:giut-injective})
holds.

Define a cocycle $c : \Gg_\Lambda \to \ZZ^k$ by $c(x,n,y):=n$.
Then $c$ is locally constant and therefore continuous. Hence
there is a strongly continuous  action $\beta: \TT^k\to
\Aut(C^*_r(\Gg_\Lambda, E_X))$ determined by $\beta_z(f)(g) =
z^{c(g)} f(g)$. Observe that $\beta_z(f^{\lambda,\mu}_T) =
z^{d(\lambda) - d(\mu)} f^{\lambda,\mu}_T$ for all $T\in
\Kk(X_\mu,X_\lambda)$. In particular, for all $v\in\Lambda^0$
and  $a\in A_v=\Kk(X_v,X_v)$ we have
\[
\beta_z(\pi_v(a))=\beta_z(f^{v,v}_{l_a})=f^{v,v}_{l_a}=\pi_v(a),
\]
and for all $\lambda\in\Lambda$ and $x\in X_\lambda$ we have
\[
\beta_z(\rho_\lambda(x))=\beta_z(f^{\lambda, s(\lambda)}_{l_x})=
z^{d(\lambda)}f^{\lambda, s(\lambda)}_{l_x}=z^{d(\lambda)}\rho_\lambda(x).
\]
Hence Theorem~\ref{thm:giut}(\ref{it:giut-gauge}) is also
satisfied. Thus $\Psi$ is injective.

For surjectivity, it suffices by Lemma~\ref{lem:fxdense} to
show that the set $\Ee_X$ is contained in the range of $\Psi$.
For this it suffices by definition of $\Ee_X$ to show that
$f^{\lambda,\mu}_T$ is in the range of $\Psi$ for all
$\lambda,\mu\in \Lambda$ with $s(\lambda)=s(\mu)$ and all $T\in
\Kk(X_\mu,X_\lambda)$. Fix such $\lambda$, $\mu$ and $T$, and
let $v := s(\lambda)=s(\mu)$. Since the range of $\Psi$ is
closed, we may assume that $T$ is finite rank, that is, $T =
\sum_{j=1}^n \theta_{x_j, y_j}$ where $x_j \in X_\lambda$ and
$y_j \in X_\mu$ for $j = 1, \dots, n$.  We have $T = \sum_{j=1}^n l_{x_j}l^*_{y_j}$.
Hence, by the linearity of the section map $S\mapsto
f^{\lambda,\mu}_S$ and equation (\ref{eq:product formula}), we
have:
\begin{align*}
f^{\lambda,\mu}_T
    &= \sum_{j=1}^n f^{\lambda,\mu}_{l_{x_j}l^*_{y_j}}
            = \sum_{j=1}^n f^{\lambda,v}_{l_{x_j}}  f^{v,\mu}_{l^*_{y_j}}
            = \sum_{j=1}^n f^{\lambda,v}_{l_{x_j}}  (f^{\mu, v}_{l_{y_j}})^* \\
    &=  \sum_{j=1}^n \Psi(\rho^X_\lambda(x_j))\Psi(\rho^X_\lambda(y_j))^*.
\end{align*}
Thus $f^{\lambda,\mu}_T$ is in the range of $\Psi$. This concludes the proof of surjectivity.
\end{proof}

\section{Examples}\label{sec:examples}

\subsection{Higher-rank graph $C^*$-algebras}\label{sec:k-graph alg
example} Let $\Lambda$ be a higher-rank graph. For each $v \in
\Lambda^0$, let $A_v$ be the 1-dimensional $C^*$-algebra $\CC$,
and for each $\lambda \in \Lambda$, let $X_\lambda$ be the
1-dimensional Hilbert space $\CC$ regarded as an
$A_{r(\lambda)}$--$A_{s(\lambda)}$ $C^*$-correspondence. The
multiplication isomorphisms $\chi_{\alpha,\beta}: w\otimes
z\mapsto wz$ determine a $\Lambda$-system $(A,X,\chi)$ of
$C^*$-correspondences. Observe that the product system $Y$ of
$C^*$-correspondences constructed from this example as in
Proposition~\ref{prp:assoc prod sys} is precisely the product
system associated to $\Lambda$
by~\cite[Proposition~3.2]{RS2005}. The $\Lambda$-system
$(A,X,\chi)$ is regular precisely when $\Lambda$ is row-finite
and has no sources, and then $C^*(A,X,\chi)$ is canonically
isomorphic to $C^*(\Lambda)$ by \cite[Theorem~4.2]{RS2005}.

Suppose that $\Lambda$ is indeed row-finite with no sources.
Then the Fell-bundle $E_X$ is the trivial bundle $\Gg_\Lambda
\times \CC$, and its reduced cross-sectional algebra is
 the completion of the image of $C_c(\Gg_\Lambda)$
under the regular representation, that is, the reduced groupoid
$C^*$-algebra $C^*_r(\Gg_\Lambda)$. The isomorphism
$C^*(\Lambda)\cong C^*_r(\Gg_\Lambda)$ of Kumjian and Pask (see
\cite[Corollary 3.5(i)]{KP}) is a special case of Theorem
\ref{thm:isomorphism}.

\subsection{Strong shift equivalent
$C^*$-correspondences.}\label{sec:MPT}

Consider the situation of
Example~\ref{ex:systems}(\ref{it:MPT}). That is, let $A$ and
$B$ be $C^*$-algebras, let $R$ be an $A$--$B$ correspondence,
and let $S$ be a $B$--$A$ correspondence. Suppose that $R$ and
$S$ are each full and nondegenerate, and that the left action
of the coefficient algebra on each is implemented by an
injective homomorphism into the compact operators.
Let $(\Lambda^0,\Lambda^1)$ be the directed graph with vertices ${v,w}$ and edges
$\{e,f\}$ where $s(e) = r(f) = v$ and $r(e) = s(f) = w$. Define
$A_v = A$ and $A_w = B$, $X_e = R$ and $X_f = S$, and for $m
\ge 2$ and $\mu \in \Lambda^m$, define
\[
X_\mu := X_{\mu_1} \otimes_{A_{s(\mu_1)}} X_{\mu_2}
\otimes_{A_{s(\mu_2)}} \cdots \otimes_{s(\mu_{m-1})} X_{\mu_m}.
\]
Define isomorphisms $\chi_{\mu,\nu} : X_\mu
\otimes_{A_{s(\mu)}} X_\nu$ by
\[
\chi_{\mu,\nu}\big((x_{\mu_1} \otimes \cdots \otimes x_{\mu_m}) \otimes (x_{\nu_1} \otimes  \cdots \otimes x_{\nu_n})\big)
    = x_{\mu_1} \otimes \cdots \otimes x_{\mu_m} \otimes x_{\nu_1} \otimes  \cdots \otimes x_{\nu_n}.
\]
The result is a regular $\Lambda$-system $(A,X,\chi)$ of
correspondences. Notice that our notion of regularity implies the one in \cite{MPT}.

Muhly, Pask and Tomforde show in \cite{MPT} that $\Oo_{R
\otimes_B S}$ and $\Oo_{S \otimes_A R}$ are Morita-Rieffel
equivalent.  We can reinterpret this in terms of the
$\Lambda$-system. Indeed, if $1_v, 1_w \in \Mm(C^*(A,X,\chi))$ are
the images of the units of the multiplier algebras of $A_v$ and $A_w$
respectively, then each of $1_v$ and $1_w$ is a full projection, and two
applications of the gauge-invariant uniqueness theorem show
that
\[
1_v C^*(A,X,\chi) 1_v \cong \Oo_{R \otimes_B S}
    \quad\text{ and }\quad
1_w C^*(A,X,\chi) 1_w \cong \Oo_{S \otimes_A R},
\]
so that $1_v C^*(A,X,\chi) 1_w$ implements the desired
Morita-Rieffel equivalence.

The graph $\Lambda$ has precisely two infinite paths, namely
$(ef)^\infty \in v\Lambda^\infty$ and $(fe)^\infty \in w\Lambda^\infty$. By
construction, 
\[
E_{(ef)^\infty} = \varinjlim_n \Kk((R \otimes_B S)^{\otimes n})
    \quad\text{ and }\quad
E_{(fe)^\infty} = \varinjlim_n \Kk((S \otimes_A R)^{\otimes n}),
\]
which are precisely the fixed-point subalgebras of the copies
of $\Oo_{R \otimes_B S}$ and $\Oo_{S \otimes_A R}$ embedded in
$C^*(A,X,\chi)$ as above.

More generally, let $\Lambda$ be the path-category of the
directed graph consisting of a simple cycle of length $n$.
Let $(A,X,\chi)$  be a regular $\Lambda$-system. For any two
vertices $v,w \in \Lambda^0$ let $\mu_{v,w}$ denote the unique
path of minimal length from $w$ to $v$, and for each vertex $v$, let
$\lambda_v$ denote the unique cycle of length $n$ with range
$v$. Proposition~\ref{prp:Ex as corner} implies that the fibre
$E_{x_v}$ is isomorphic to the fixed point
algebra $\Oo_{X_{\lambda_v}}^\gamma$ for the gauge action on
the Cuntz-Pimsner algebra of $X_{\lambda_v}$ (see the final
paragraph of Section~\ref{sec:Pimsner}). Then for distinct $v,w$, the correspondences
$X_{\mu_{v,w}}$ and $X_{\mu_{w,v}}$ become an instance of the
situation considered above, and we obtain a Morita-Rieffel
equivalence
\begin{equation}\label{eq:loop MREs}
\Oo_{X_{\lambda_v}} \cong \Oo_{X_{\mu_{v,w}} \otimes_{A_w} X_{\mu_{w,v}}} \sim
    \Oo_{X_{\mu_{w,v}} \otimes_{A_v} X_{\mu_{v,w}}} \cong \Oo_{X_{\lambda_w}}.
\end{equation}
Indeed, as argued above, each $1_v C^*(A, X, \chi) 1_v \cong
\Oo_{X_{\lambda_v}}$ and the $1_v  C^*(A, X, \chi) 1_w$
implement the Morita-Rieffel equivalences~\eqref{eq:loop MREs}.
For each vertex  $v \in \Lambda^0$ there is a unique infinite path
$x_v=(\lambda_v)^\infty$ such that $v = r(x_v)$.
Recall that we may identify $\Lambda^\infty$ with $\Gg_\Lambda^{(0)}$, and then
for each  $v \in \Lambda^0$  the fibre $E_{x_v}$  is isomorphic to the
fixed-point algebra for the gauge action on
$\Oo_{X_{\lambda_v}}$.

\subsection{$\Gamma$-systems of $k$-morphs.}
Let $\Gamma$ be a row-finite $\ell$-graph with no sources, and let
$W$ be a $\Gamma$-system of $k$-morphs satisfying the technical
assumptions (\maltese) of \cite{KPS2}. Proposition~6.4 of
\cite{KPS2} shows how to associate to each $k$-morph $W_\gamma$
(where $\gamma \in \Gamma$) a
$C^*(\Lambda_{r(\gamma)})$--$C^*(\Lambda_{s(\gamma)})$
correspondence $\Hh(W_\gamma)$. The proof of
\cite[Theorem~6.6]{KPS2} shows how to construct isomorphisms
$\chi_{\mu,\nu} : \Hh(W_\mu) \otimes_{C^*(\Lambda_{s(\mu)})}
\Hh(W_\nu) \to \Hh(W_{\mu\nu})$, and it is routine to verify
using the associativity conditions imposed on the system $W$ of
$k$-morphs that the $\chi_{\mu,\nu}$ satisfy the associativity
condition~(\ref{it:sys assoc}) of Definition~\ref{dfn:Lambda
system}. Let $A_v := C^*(\Lambda_v)$ for all $v \in \Gamma^0$ and
let $X_\gamma := \Hh(W_\gamma)$ for each $\gamma \in \Gamma$.
Then $(A, X, \chi)$ is a $\Gamma$-system of
$C^*$-correspondences.

Let $\Sigma$ be a $\Gamma$-bundle for the $\Gamma$-system of
$k$-morphs. This is a $(k+ \ell)$-graph $\Sigma$ together with a
functor $f : \Sigma \to \Gamma$ called the bundle map such that
the degree $d(f(\lambda))=(d(\lambda)_{k+1}, \dots,
d(\lambda)_{k+ \ell})$  for all $\lambda \in
\Sigma$, each $f^{-1}(v) \cong \Lambda_v$, and each
$f^{-1}(\lambda) \cap \{\mu \in \Sigma : d(\mu)_i = 0\text{ for
$i \le k+1$}\}$ is isomorphic to $X_\lambda$. We claim that
$C^*(\Sigma) \cong C^*(A,X,\chi)$. To see this, fix $\sigma \in
\Sigma$. Factorise $\sigma = y\lambda$ where $d(\lambda)_i = 0$
for $i \ge k+1$ and $d(y)_i = 0$ for $i \le k$. Then we may
identify $\lambda$ with an element of $\Lambda_{s(\gamma)}$ and
$y$ with an element of $W_{f(\sigma)}$. Let $\gamma=f(\sigma)$,
so that $y \in W_\gamma$.

Proposition~6.7 of \cite{KPS2} shows that
\[
X_\gamma = \Hh(W_\gamma) \cong \overline{C_c(W_\gamma) \otimes_{C_0(\Lambda_{s(\gamma)}^0)} C^*(\Lambda_{s(\gamma)})}.
\]
In particular if $\delta_y \in C_c(W_\gamma)$ is the
point-mass, and if $s_\lambda$ denotes the canonical generator
of $C^*(\Lambda_{s(\gamma)})$, 
then
$\delta_y \otimes s_\lambda$ is an element of $X_\gamma$. Let
$(\rho, \pi)$ denote the universal representation of $(A, X,
\chi)$ in $C^*(A,X,\chi)$, and define $t_\sigma \in
C^*(A,X,\gamma)$ by $t_\sigma := \rho_\gamma(\delta_y \otimes
s_\lambda)$. It is routine to check that $\{t_\sigma : \sigma
\in \Sigma\}$ is a Cuntz-Krieger $\Sigma$-family. These
elements generate $C^*(A,X,\chi)$ because the elements
$\delta_y \otimes s_\lambda$ span a dense subspace of each
$X_\gamma$. Hence the universal property of $C^*(\Sigma)$
ensures that there is a surjective homomorphism $\phi :
C^*(\Sigma) \to C^*(A,X,\chi)$ such that $\phi(s_\sigma) =
t_\sigma$ for all $\sigma$. A straightforward application of
the gauge-invariant uniqueness theorem \cite[Theorem~3.4]{KP}
shows that $\phi$ is injective.

For $g = (\alpha x, d(\alpha) - d(\beta), \beta x) \in
\Gg_\Gamma$, the fibre $E_g$ of the Fell-bundle $E$
 can be described using the
isomorphism $\phi : C^*(\Sigma) \to C^*(A,X,\chi)$ as
\[\textstyle
E_g = \clsp\Big(\bigcup_{n \in \NN^k} \{\phi(s_\mu  s^*_\nu) :
        f(\mu) = \alpha x(0,n) \text{ and } f(\nu) = \beta x(0,n)\}\Big).
\]

\subsection{Systems of endomorphisms of a single $C^*$-algebra}

Consider a $C^*$-algebra $A$ and a row-finite $k$-graph
$\Lambda$ with no sources. Suppose that $\lambda \mapsto
\varphi_\lambda$ is a contravariant functor from $\Lambda$ to
$\End_1(A)$, the semigroup  of  approximately unital  endomorphisms
of $A$ regarded as a category with one object $A$.
Assuming each $\varphi_\lambda$ is injective,
we construct a regular $\Lambda$-system $(A,X,\chi)$
as
in Example \ref{ex:systems}(\ref{ex:endo}), where
 $A_v = A$ for  $v \in \Lambda^0$ and  $X_\lambda={_{\varphi_\lambda} A}$ for $\lambda\in\Lambda^1$.
 In this case,  $E_x=\varinjlim(A_{x(n)}, \varphi_{x(m,n)})$
 for $x\in\Lambda^\infty={\mathcal G}_\Lambda^{(0)}$; if all $\varphi_\lambda$ are automorphisms,
 $E_x$ is isomorphic to $A$ for all $x$.

We regard the
$C^*$-algebra $C^*(A,X,\chi)$  as a kind of
crossed-product of $A$ by an action  of $\Lambda$. We consider two special cases (see also section 4.3 in \cite{PWY}).

\begin{example}[Crossed products by $\ZZ^k$]
Consider $k$ commuting automorphisms $\alpha_1,...,\alpha_k$ of a $C^*$-algebra $A$.
For $m \in \mathbb{N}^k$ let $X_m =
{}_{\alpha^m}A$, where $\alpha^m=\alpha_1^{m_1} \cdots \alpha_k^{m_k}$. Then $X = (X_m)$ is a
$T_k$-system of correspondences, where $T_k$ is the $k$-graph
whose path category can be identified with $\mathbb{N}^k$. The
corresponding $C^*$-algebra $C^*(A,X,\chi)$ is isomorphic to the crossed
product $A \rtimes {\mathbb Z}^k$.

In this instance, the groupoid $\Gg_\Lambda$ is isomorphic to
$\ZZ^k$.  Note that $E_0=A$, and, more generally, for $n\in \ZZ^k$ we have $E_n={}_{\alpha^n}A$
with the same multiindex convention as above.
Observe that the Fell bundle $E$ is the Fell bundle over $\ZZ^k$
obtained in the usual fashion from the dual action of $\TT^k$
on the crossed product.
\end{example}

\begin{example}[Cuntz's twisted tensor products]
In \cite{Cu}, Cuntz constructs a twisted tensor product
$A\times_{\mathcal U}{\mathcal O}_n$. Let $A$ be a unital $C^*$-algebra
acting nondegenerately on the Hilbert space $\mathcal H$, let
 ${\mathcal U}=(U_1,...,U_n)$ be a family of commuting unitaries
in ${\mathcal L}({\mathcal H})$ implementing automorphisms
$\alpha_1,...,\alpha_n$ of $A$, and let $S_1,...,S_n$ be the isometries
generating ${\mathcal O}_n$.
Cuntz defines $A\times_{\mathcal U}{\mathcal O}_n$ as the $C^*$-subalgebra
of ${\mathcal L}({\mathcal H})\otimes{\mathcal O}_n $ generated
by $A\otimes 1$ together with $U_1\otimes S_1,...,U_n\otimes
S_n$.
 Notice that we have the relations
\[
(\alpha_i(a)\otimes 1)(U_i\otimes S_i)=(U_i\otimes S_i)(a\otimes 1),\;\text{for}\; a\in A,\; i=1,...,n.
\]

Let $\Lambda$ be the $1$-graph with $\Lambda^0 = \{v\}$ and
$\Lambda^1 = \{e_1,...,e_n\}$. Let $\alpha_1, \dots, \alpha_n$
be any collection of automorphisms of $A$ (we need not assume
that the $\alpha_i$ commute, though this is the case in Cuntz's
setting). Let $A_v := A$,  and for each edge $e_i$ let $X_{e_i}
:= {_{\alpha_i^{-1}} A}$. Let $(A,X,\chi)$ be the corresponding
$\Lambda$-system. The $C^*$-algebra $C^*(A,X,\chi)$ is
generated by $A$ and $n$ isometries $V_1,...,V_n$ with relations
\[
V_i^*V_j=\delta_{ij},\;\; \sum_{i=1}^nV_iV_i^*=1,\;\; \alpha_i(a)V_i=V_ia,\;\; i=1,...,n.
\]
Indeed, for $i\le n$, let $x_i=(0,...,1_A,...,0)\in
\bigoplus_{i=1}^nX_{e_i}$, where the $1_A$ is in position $i$. Then $\langle x_i,
x_j\rangle=\delta_{ij}$ and $x=\sum_{i=1}^nx_i\langle x_i,
x\rangle$ for all $x\in \bigoplus_{i=1}^nX_{e_i}$. Consider the image $V_i$ of $x_i$
in $C^*(A,X,\chi)$. Then $V_i^*V_j=\delta_{ij}$, $\sum_{i=1}^n V_iV_i^*=1$, and
since $a\cdot x_i=x_i\langle x_i, a\cdot x_i\rangle$, we get $aV_i=V_i\alpha_i^{-1}(a)$
or $\alpha_i(a)V_i=V_ia$ for $a\in A$. If the $\alpha_i$ commute, the isomorphism
between $C^*(A,X,\chi)$ and
$A\times_{\mathcal U}{\mathcal O}_n$ is given by
\[
a\mapsto a\otimes 1,\;\; V_i\mapsto U_i\otimes S_i.
\]
Notice that $C^*(A,X,\chi)$ unitally contains an isomorphic
copy of $C^*(\Lambda)={\mathcal O}_n$, and the groupoid
$\Gg_\Lambda$ is the Cuntz groupoid  described in
\cite[Definition~III.2.1]{Renault1980}. As noted above, each fibre of the  Fell bundle
over $\Gg_\Lambda^{(0)}$ is isomorphic to $A$.
\end{example}

\subsection{Ionescu's $C^*$-algebras associated to
Mauldin-Williams graphs}\label{sec:Mauldin-Williams}
Our discussion is based on the class of examples considered
by Ionescu in \cite{Ion}.
Let $\Lambda$ be a $k$-graph (Ionescu considers a finite
directed graph).
Let $\LCHS$ be the category whose
objects are compact metric spaces $T$, and whose
morphisms are contractions. Let $\lambda \mapsto
\varphi_\lambda$ be a covariant functor from $\Lambda$ to
$\LCHS$; since we identify the vertices of $\Lambda$ with its
objects, we have $\varphi_v = \id_{T_v}$ for each $v \in
\Lambda^0$.

For $v \in \Lambda^0$, let $A_v := C(T_v)$, and for $\lambda
\in \Lambda$, let $\varphi^*_\lambda : A_{r(\lambda)} \to
A_{s(\lambda)}$ be the induced map $\varphi^*_\lambda(f) = f
\circ \varphi_\lambda$. Now define $X_\lambda$ to be the
$C^*$-correspondence ${_{\varphi^*_\lambda} A_{s(\lambda)}}$
from $A_{r(\lambda)}$ to $A_{s(\lambda)}$. For each composable
pair $\alpha,\beta$ in $\Lambda$, there is an isomorphism
$\chi_{\alpha,\beta} : X_{\alpha} \otimes_{A_{s(\alpha)}}
X_\beta \to X_{\alpha\beta}$ given by
\[
\chi_{\alpha,\beta}(f \otimes_{A(s(\alpha))} g)(t) = (\varphi^*_\beta(f) g)(t) = f(\varphi_\beta(t))g(t).
\]
Since $T \mapsto C(T)$ and $\varphi \mapsto \varphi^*$
determines a contravariant functor from $\LCHS$ to the category
of $C^*$-algebras with $C^*$-homomorphisms as morphisms, the triple
$(A,X,\chi)$ is a $\Lambda$-system of $C^*$-correspondences.

The $X_\lambda$ are always full and nondegenerate with a left
action by compact operators. The left actions are all injective
 when the $\varphi_\lambda$ are all surjective. In
particular, the system $X$ is regular if $\Lambda$ is
row-finite with no sources, and each $\varphi_\lambda$ is
surjective.

For each element $x \in \Lambda^\infty$, the fibre $E_x$ of the
Fell bundle $E$ is given by
\[
E_x = \varinjlim \Kk(X_{x(0,n)}) = \varinjlim(C(T_{x(n)}), \varphi^*_{x(m,n)})
= C(\varprojlim(T_{x(n)}, \varphi_{x(m,n)})).
\]

In Ionescu's setting, $\Lambda$ is the path category of a finite
directed graph with no sinks or sources and there is a constant
 $c < 1$, such that the contraction constant $c_\lambda$
 of $\varphi_\lambda$ satisfies $c_\lambda \le c$ for
every $\lambda \in \Lambda^1$. As in
Remark~\ref{rmk:graph system}, the functor is then determined by
$\{T_v : v \in \Lambda^0\}$ and $\{\varphi_e : e \in \Lambda^1\}$. Since
$c_\lambda \le c< 1$ for every $\lambda \in \Lambda^1$,  the intersection
$\bigcap^\infty_{n=1} \varphi_{x(0,n)}(T_{x(n)})$ is a
singleton $\{t_x\}$ for any infinite path
$x \in \Lambda^\infty$. By the universal property of projective
limits, it follows that $\varprojlim(T_{x(n)},
\varphi_{x(m,n)})$ is a singleton, so $E_x \cong \CC$.
Unfortunately, the $\varphi_{\lambda}$ will typically not be
surjective in this setting. So  the resulting $\Lambda$-system
will fail to be regular.

Quigg (see \cite{Q}) considers graphs of  $C^*$-correspondences
over row-finite $1$-graphs with no sources with a view towards
generalizing Ionescu's result by working in the category of locally
compact Hausdorff spaces with continuous proper maps.

\subsection{Pinzari, Watatani and Yonetani's 
systems of $C^*$-correspondences}

In section 5.3 of \cite{PWY}, Pinzari, Watatani and Yonetani
study KMS states on the $C^*$-algebra defined using a finite family of
 $C^*$-correspondences. More precisely, let $A_1, \dots, A_n$
be unital simple $C^*$-algebras. Fix a matrix
$\Sigma=(\sigma_{i,j})\in M_n(\{0,1\})$ with no row and no
column identically zero. For each pair $(i,j)$ such that
$\sigma_{i,j}=1$, let $X_{i,j}$ be a full, finite projective
$A_i$-$A_j$ $C^*$-correspondence. Pinzari, Watatani and Yonetani
study the KMS states on the Cuntz-Pimsner algebra of the
$C^*$-correspondence $X := \bigoplus_{\sigma_{i,j} \not= 0}
X_{i,j}$ over $A := \bigoplus^n_{i=1} A_i$.

Let $\Lambda_{\Sigma}$ be the $1$-graph with $\Lambda^0 = \{v_1,
\dots, v_n\}$ and $\Lambda^1 = \{e_{i,j} : \sigma_{i,j} = 1\}$ with
$s(e_{i,j}) = v_j$ and $r(e_{i,j}) = v_i$. As in
Remark~\ref{rmk:graph system}, setting $A_{v_i} := A_i$ and
$X_{e_{i,j}} := X_{i,j}$ determines a regular $\Lambda_{\Sigma}$-system of
$C^*$-correspondences. Our construction of $C^*(A,X,\chi)$ in
Definition~\ref{dfn:C^*(X)} is so that $C^*(A,X,\chi) \cong
\Oo_{\bigoplus X_i}$, the $C^*$-algebra studied by
Pinzari-Watatani-Yonetani.

\subsection{Topological graphs fibred over directed graphs}

Let $\Lambda$ be a  $1$-graph, and fix  locally
compact spaces $T_v$ and $U_e$ for each $v\in \Lambda^0$ and $e\in \Lambda^1$. Suppose there are
local homeomorphisms $\sigma_e:U_e\to T_{s(e)}$ and continuous  maps
$\rho_e:U_e\to T_{r(e)}$ for each $e\in\Lambda^1$. These data define a  $\Lambda$-system
of $C^*$-correspondences, by taking $A_v=C_0(T_v)$ and $X_e$ to be the $A_{r(e)}$--$A_{s(e)}$
$C^*$-correspondence obtained from the
completion
of $C_c(U_e)$, with  inner product  defined by
\[
\langle \xi,\eta\rangle(t)=\sum_{\sigma_e(u)=t}\overline{\xi(u)}\eta(u)
\]
and with right and left multiplications defined by
\[
(f\cdot\xi\cdot g)(u)=f(\rho_e(u))\xi(u)g(\sigma_e(u)),
\]
where $\xi, \eta\in C_c(U_e)$, $f\in C_0(T_{r(e)})$, and
$g\in C_0(T_{s(e)})$. This $\Lambda$-system is regular if $\Lambda$ is row-finite
with no sources, the maps $\sigma_e$ are surjective, and the $\rho_e$ are proper
with dense range.
 By construction, $C^*(A,X,\chi)$ is isomorphic to the $C^*$-algebra of the topological graph
with vertex space $\bigsqcup_{v\in \Lambda^0} T_v$, edge space
$\bigsqcup_{e\in \Lambda^1} U_e$, and source and range maps defined using $\sigma_e$ and $\rho_e$. An important special case of this example is given by
the following.

\subsection{Katsura's realisation of nonunital Kirchberg algebras}

Building on earlier work of Deaconu (see \cite{De1}), Katsura (see \cite{Katsura2008})
constructs the topological graph $\Lambda\times_{n,m}{\mathbb T}$  as above,
with $T_v=U_e=\TT$ for all $v\in \Lambda^0$ and $e\in \Lambda^1$. Given two maps
$n:\Lambda^1\to {\mathbb Z}_+$ and
$m:\Lambda^1\to {\mathbb Z}$, define $\sigma_e:U_e\to T_{s(e)}$ by
$\sigma_e(z)=z^{n(e)}$, and $\rho_e:U_e\to T_{r(e)}$ by
$\rho_e(z)=z^{m(e)}$. He shows that
every nonunital Kirchberg algebra is isomorphic to the
$C^*$-algebra of such a topological graph where  the maps $n,m$
are chosen appropriately (see \cite[Lemmas 4.2~and~4.4 and
Proposition~4.5]{Katsura2008}).
The associated $\Lambda$-system of $C^*$-correspondences  is regular precisely when $\Lambda$
is row-finite with no sources and $m(e)\neq 0$ for all $e\in \Lambda^1$.

\end{document}